\documentclass{article}

\usepackage[utf8]{inputenc}
\usepackage{authblk}
\usepackage{hyperref}
\usepackage{enumitem}

\usepackage{array}

\usepackage{tikz}
\usepackage{tikz-cd}
\usetikzlibrary{arrows}
\tikzcdset{arrow style=tikz,
diagrams={>={Stealth[round,length=4pt,width=6pt,inset=3pt]}}}

\usepackage{amsthm}

\usepackage{amsmath} 
\usepackage{amssymb}
\usepackage{amsfonts}
\usepackage{faktor}
\usepackage[T1]{fontenc}
\usepackage{faktor}
\usepackage{xfrac}
\usepackage{color}
 \usepackage{todonotes}
\usepackage[english]{babel}
 
\newtheorem{theorem}{Theorem}
\newtheorem{definition}[theorem]{Definition}

\newtheorem{lemma}[theorem]{Lemma}
\newtheorem{notation}[theorem]{Notation}

\newtheorem{proposition}[theorem]{Proposition}

\newtheorem{example}[theorem]{Example}
\newtheorem{remark}[theorem]{Remark}
\newtheorem{corollary}[theorem]{Corollary}

\newtheorem*{theorem*}{Theorem}

\DeclareMathOperator{\ord}{ord}
\DeclareMathOperator{\length}{length}
\DeclareMathOperator{\Hess}{Hess}
\DeclareMathOperator{\Mat}{Mat}

\DeclareMathOperator{\Frob}{Frob}

\newcommand{\equivr}{\overset{\mathcal{R}}{\sim}}
\newcommand{\equivk}{\overset{\mathcal{K}}{\sim}}
\newcommand{\equivd}{\overset{I}{\sim}}

\newcommand{\maxV}{\mathfrak{m}_V}
\newcommand{\naxV}{\mathfrak{n}_V}
\newcommand{\Vx}{V[[\underline{x}]]}
\newcommand{\Vxy}{V[[\underline{x},y]]}
\newcommand{\pder}{\frac{\partial}{\partial \pi}}
\newcommand{\kappax}{\kappa[[\underline{x}]]}

\numberwithin{theorem}{section}

\renewcommand{\mod}{\,\operatorname{mod}\,}

\title{Detecting and Quantifying Isolated Singularities over Discrete Valuation Rings}
\author{Yotam Svoray}
\date{}

\AtEndDocument{\bigskip{\footnotesize%
  \textsc{ Department of Mathematics, University of Utah, UT 84112} \par  
  \textit{E-mail address}: \texttt{svoray@math.utah.edu} 
}}

\begin{document}

\maketitle

\begin{abstract}
        We develop a theory of isolated hypersurface singularities in mixed characteristic $(0,p)$, focusing on quotient rings over a Discrete Valuation Ring (DVR). We introduce and study analogues of the classical Tjurina and Milnor numbers for this setting, prove a generalized analogue of the Determinacy Theorem and the Mather-Yau Theorem for complete Noetherian local rings, and define numerical invariants that provide distinct criteria for detecting isolated singularities in the unramified and ramified cases. 
\end{abstract}

  \tableofcontents

%%% -*-LaTeX-*-

\section{Introduction}\label{Introduction}

The studies of hypersurfaces with isolated singularities originated with the foundational work of Arnold in~\cite{varchenko1985singularities}, Milnor in~\cite{milnor2016singular}, and Tjurina in~\cite{tjurina1969locally} in the complex setting. Their pioneering efforts involved the computation and analysis of hypersurface singularities using numerical invariants, specifically the Milnor and Tjurina numbers, along with key concepts like equivalences and determinacy. This framework was later generalized by Greuel and Kr\"{o}ning in~\cite{greuel1990simple} to algebraically closed fields of arbitrary characteristic, with subsequent expansions by Boubakri in~\cite{boubakri2009hypersurface} and Nguyen in~\cite{Nguyen2013classification}.\\

This chapter continues this line of generalization by extending the theory of isolated singularities of hypersurfaces from algebraically closed fields to a more general setting of hypersurfaces over the ring of power series over a local ring. Specifically, we focus on quantifying and understanding quotient rings $\frac{V[[\underline{x}]]}{\langle f \rangle}$, where $V$ is a Discrete Valuation Ring (DVR) with a chosen uniformizer $\pi$, with isolated singularities, i.e. such that $\left(\frac{\Vx}{\langle f \rangle}\right)_\mathfrak{q}$ is regular for every non-maximal prime ideal $\mathfrak{q}$. The key challenge lies in developing numerical invariants, analogous to the classical Tjurina and Milnor numbers, that can effectively detect and measure the severity of these singularities.\\

When moving to the DVR case, there is a constraint that we do not encounter when working over fields. In this case,  the uniformizer $\pi$ of $V$ is non invertible, and behaves differently than the units in $V$. This causes the groups of $V$-automorphisms of the ring of formal power series over $V$ to be much smaller compared to its field counterpart. Therefore elements $f,g \in V[[\underline{x}]]$ that we would want to be considered as "equivalent up to a change of variables" can not be. For example, as in Example 1.1. in~\cite{carvajal2019covers}, one would want to view  $x_1^2 + x_2^2 + \pi^3$ and $x_1^3 + x_2^2 + \pi^2$ as equivalent.  But we have that
\begin{equation*}
    \frac{V[[x_1, x_2]]}{\langle x_1^2 + x_2^2 + \pi^3 \rangle} \not \cong \frac{V[[x_1, x_2]]}{\langle x_1^3 + x_2^2 + \pi^2 \rangle},
\end{equation*}
\noindent as otherwise$\mod \pi$ they would be isomorphic over the field $\frac{V}{\langle \pi \rangle}$, which is impossible since the first would define a cusp singularity and the second an ordinary double point.  \\

In order to avoid this constraint, we would like to view the uniformizer $\pi$ as an additional variable, which allows us to view formal power series over $V$ in variables $x_1, \dots, x_n$ as "power series in $\pi, x_1, \dots, x_n$ over $\frac{V}{\langle \pi \rangle}$" and since $\pi$ behaves differently than the other variables, we would then like to formally replace $\pi$ by a dummy variable $y$. In the example above, one has that
\begin{equation*}
    \frac{V[[x_1, x_2,y]]}{\langle x_1^2 + x_2^2 + y^3 \rangle} \cong \frac{V[[x_1, x_2,y]]}{\langle x_1^3 + x_2^2 + y^2 \rangle}
\end{equation*}
by "replacing $\pi$ by the dummy variable $y$".\\

Over an arbitrary DVR, it is not clear how we can do this "replacing" procedure formally. For example, if we look at $1+\pi + \pi x$, since $1+\pi$ is a unit in $V \subset V[[\underline{x}]]$, then we can write 
\begin{center}
\vspace{-10mm}
\begin{equation*}
    1+\pi + \pi x = 1\cdot \pi^0x^0 + 1 \cdot \pi^1 x^0 + 0 \cdot \pi^0x^1 + 0 \cdot \pi^2x^0 + 1\cdot \pi^1 x^1+ \cdots 
\end{equation*}
    and
    \begin{equation*}
    (1+\pi) + \pi x = (1+\pi)\cdot \pi^0x^0 + 0 \cdot \pi^1 x^0 + 0 \cdot \pi^0x^1 + 0 \cdot \pi^2x^0 + 1\cdot \pi^1 x^1 + \cdots.
\end{equation*}
\end{center}

\noindent Therefore, intuitively,  there are two different ways to "replace $\pi$ by $y$", which are 
\begin{center}
  \vspace{-10mm}
\begin{equation*}
    1+\pi + \pi x  \rightsquigarrow 1\cdot y^0x^0 + 1 \cdot y^1 x^0 + 0 \cdot y^0x^1 + 0 \cdot y^2x^0 + 1\cdot y^1 x^1 + \cdots 
\end{equation*}
    or
    \begin{equation*}
   (1+\pi) + \pi x  \rightsquigarrow (1+\pi)\cdot y^0x^0 + 0 \cdot y^1 x^0 + 0 \cdot y^0x^1 + 0 \cdot y^2x^0 + 1\cdot y^1 x^1+ \cdots.
\end{equation*}
\end{center}
\noindent Yet, there is a way to systematically "replace $\pi$ by a dummy variable $y$", by first looking at the expansion with respect to $x$ as a power series, and then looking at its coefficients (as we formally define in Notation~\ref{def:subsoil}), recalling that since $V$ is a DVR, then every element in $V$ can be written uniquely as a unit times a power of $\pi$.\\

In the example above,  since we first look at the expansion with respect to $x$, then 
\begin{equation*}
    (1+\pi) + \pi x = (1+\pi)\cdot x^0 + \pi \cdot x^1 + 0 \cdot x^2 + \cdots, 
\end{equation*}
 and since $(1+\pi)=(1+\pi) \cdot \pi^0$ and $\pi = 1 \cdot \pi^1$ as elements of $V$, then we would replace $1+\pi + \pi x$ by $(1+\pi) + y x$.\\

This point of view allows us to formally define an analogue of the Tjurina and Milnor numbers over $V$ that behave similarly their field counterparts.  In addition, it allows us to define equivalences $\sim$ of hypersurface singularities (defined formally in Definition~\ref{def:equiv}). Using the notion of equivalence and the Tjurina number over $V$, we develop a theory of hypersurface singularities over $V$, analogous to the results over fields, as summarized in Section 2 of~\cite{greuel2007introduction}.  Our main results include the following:
\begin{enumerate}
    \item We present a generalization of the Determinacy Theorem (Section \ref{appendix_Det}) to complete Noetherian local rings, inspired by Pellikaan \cite{pellikaan1988finite} and Boubakri-Greuel-Markwig in~\cite{boubakri2012invariants}. This result is then used to prove a generalized Mather-Yau theorem and to define a new numerical invariant for singularities over complete local rings that measures the determinacy.
    \item We define analogues of the Tjurina number and contact equivalence over $V[[\underline{x}]]$ (Section \ref{sec:V_power_series}) by viewing the uniformizer of $V$ "as a variable", and proving they satisfy properties similar to their field analogues, including the splitting lemma and Morse lemma.
    \item We define numerical invariants (in Section \ref{sec:Jeffries_Hochter}) to detect isolated singularities based on whether $V$ is ramified or unramified, drawing on recent works by Hochster and Jeffries in~\cite{hochster2021jacobian}, Saito in~\cite{saito2022frobenius}, and KC in~\cite{kc2024singular}.
\end{enumerate}

\medskip

We summarize these results in the following Theorem: 

\begin{theorem}
\begin{enumerate}
    \item Let  $R$ be a complete Noetherian local ring, let $I \subset R[[\underline{x}]]$ be an ideal, let $f, g \in I^2 \subset R[[\underline{x}]]$
    \begin{enumerate}
        \item (Theorem~\ref{thm_App})  If $f$ satisfies
    \begin{equation*}
        I^{k+2} \subset I \cdot \langle f \rangle + I^2 \cdot J(f),
    \end{equation*}
      then $f$ is $\left(2k-\ord_I\left(f\right)+2\right)-$determined with respect to $I$ (as in Definition~\ref{def:determinacy}).
      \item (Proposition~\ref{prop:det_other_direction}) If $f$ has finite determinacy with respect to $I$ then $\textup{Sing}(V(f)) \subset V(I + \langle f \rangle + J(f))$.
      \item (Proposition~\ref{prop:mather-yau}) Assume that $f$ has finite determinacy with respect to $I$. Then $f \equivd g$ (as in Definition~\ref{def:determinacy}) if and only if for every $k\gg 0$ we have an isomorphism of $R-$algebras
    \begin{equation*}
        \frac{R[[\underline{x}]]}{\langle f, I^k\cdot  J(f) \rangle} \cong \frac{R[[\underline{x}]]}{\langle g, I^k\cdot  J(g)\rangle}.
    \end{equation*}
    \end{enumerate}
    \item Let $(V, \langle \pi\rangle, \kappa)$ be a complete DVR and let $f \in \Vx$. Then
    \begin{enumerate}
        \item (Proposition~\ref{prop:milnor_zero}) $\tau_V(f)=0$ (as in Definition~\ref{def:tjurinaV}) if and only if $f \sim x_1$  if and only if $\frac{\Vx}{\langle f \rangle}$ is a regular local ring. 
        \item (Proposition~\ref{prop:Morse}) If the characteristic of $\kappa$ is not $2$ and $\kappa$ is quadratically closed then $\tau_V(f)=1$ if and only if $f \sim \pi^2 + x_1^2 + \cdots +x_n^2$. 
        \item (Proposition~\ref{prop:finite_tjurina}, Corollary~\ref{prop:upper_bound_tau}) If $V$ is unramified then $f$ defines an isolated singularity if and only if $p \in \sqrt{J(f)}$ and $\tau(f, \delta) < \infty$ for every $p-$derivation $\delta$ (as in Definition~\ref{def:tjurina_delta}) if and only if $p \in \sqrt{J(f)}$ and $\tau^\Delta(f) < \infty$ (as in Corollary~\ref{prop:upper_bound_tau}). 
        \item (Proposition~\ref{prop:ramified_iso_sing}) If $V$ is ramified then $f$ defines an isolated singularity if and only if $\pi \in \sqrt{J(f)}$ and $\tau^\pi(f) <\infty$ (as in Definition~\ref{def:pitjurina}). 
    \end{enumerate}
\end{enumerate}
\end{theorem}

As we see in the theorem above, there are four invariants over $V$ that are analogues of the Tjurina number (namely $\tau_V(f), \tau(f, \delta), \tau^\Delta(f)$, and $\tau^\pi(f)$), each playing a different role in understanding the hypersurface $\frac{V[[\underline{x}]]}{\langle f \rangle}$. The first, $\tau_V(f)$, is defined in Definition~\ref{prop:generic_intersection} as the generic intersection of the ideal $\langle \tilde{f} \rangle + J(\tilde{f})$, where $\tilde{f}$ is the element of $V[[\underline{x},y]]$ associated with $f$ where we systematically "replaced $\pi$ by variable $y$" (as in Notation~\ref{def:subsoil}). This invariant has the most properties that are similar to the Tjurina number as defined over fields. Yet, unlike as the field case, it does not fully detect isolated singularities, but a weaker behavior that is somewhat similar. \\

The second, $\tau(f, \delta)$, is defined in Definition~\ref{def:tjurina_delta} over unramified DVRs with respect to a chosen $p-$derivation $\delta$, which we can view as a operator that "derives with respect to $p$". $\tau(f, \delta)$ detects isolated singularities completely (with an additional technical condition), yet it behaves very differently than the Tjurina number as defined over fields, in addition to highly depending on the choice of $\delta$. The third, $\tau^\Delta(f)$, is the value of $\tau(f, \delta)$ when going$\mod p$, does not depend on the choice of $\delta$ and behaves similarly to $\tau(f, \delta)$.\\

The fourth, $\tau^\pi(f)$, is defined in Definition~\ref{def:pitjurina} as an analogue of $\tau(f, \delta)$ for some ramified DVR with uniformizer $\pi$, in which case we do not need to depend on $p-$derivaties, as in the ramified case we have a derivation that acts as "taking derivative with respect to $\pi$". All of these Tjurina numbers play different roles in quantifying detecting, and understanding isolated singularities over $V$ and similar behaviors that elements in $V[[\underline{x}]]$ may exhibit. \\

\noindent\textbf{Notations}: Given a local ring $(R, \mathfrak{m}, \kappa)$, we denote by $R[[\underline{x}]]=R[[x_1, \dots, x_n]]$ the ring of power series over $R$ in variables $x_1, \dots, x_n$ and we denote its maximal ideal by $\mathfrak{M}=\mathfrak{m}+\langle x_1, \dots, x_n \rangle$. Note that $R[[\underline{x}]]$ is a complete topological ring with respect to the $\mathfrak{M}-$adic topology.  Given an ideal $\mathfrak{a} \subset R[[\underline{x}]]$, we denote by $V(\mathfrak{a})$ the set of all prime ideals $\mathfrak{a} \subset \mathfrak{p}$. Given $f,g \in R[[\underline{x}]]$, we denote $f \equivk g$ if we have an isomorphism of $R-$algebras $\frac{R[[\underline{x}]]}{\langle f \rangle} \cong \frac{R[[\underline{x}]]}{\langle g \rangle}$, and we denote $f \equivr g$ if there exists some $R-$automorphism $\varphi$ such that $\varphi(f)=g$. Given some $f$ and some variable $z$, we denote by $\partial_z(f)$ the partial derivative of $f$ with respect to $z$, we denote $\partial_i(f)=\partial_{x_i}(f)$, and we denote $J(f) = \langle \partial_1(f), \dots, \partial_n(f) \rangle$. Given an ideal $I$, we denote by $\ord_I(f)$ the largest $k \geq 0$ such that $f \in I^k$ and we denote $\ord(f)=\ord_\mathfrak{M}(f)$. In addition, given some module $M$ over $R$, we denote by $\textup{length}_{R}\left(M\right)$ the length of the module $M$ over $R$. When discussing power series, we use multi-index notation (as in~\cite{greuel2007introduction}): If $\alpha = (\alpha_1, \dots, \alpha_n)$ is a tuple of non negative integers and $\underline{x}=(x_1, \dots, x_n)$ is a tuple of ring elements, we denote $|\alpha|=\alpha_1+\cdots +\alpha_n$, we denote $\underline{x}^\alpha =  x_1^{\alpha_1} \cdots x_n^{\alpha_n}$, and we denote $\sum_{\alpha} a_{\alpha} \underline{x}^{\alpha} = \sum_{\alpha_1, \dots, \alpha_n \geq 0} a_{\alpha_1, \dots, \alpha_n} x_1^{\alpha_1} \cdots x_n^{\alpha_n}$. Given $\alpha=(\alpha_1, \dots, \alpha_n)$ and $\beta=(\beta_1, \dots, \beta_n)$, we say that $\beta \geq \alpha$ if $\beta_j \geq \alpha_j$ for every $j=1, \dots, n$. Given $f(\underline{x}) = \sum_\alpha a_\alpha \underline{x}^\alpha$, we denote $f(\underline{0})=a_0$.\\

\noindent\textbf{Acknowledgments.} This work was done as part of the Author's PhD thesis under the guidance of Karl Schwede, and we wish to thank him for his guidance, help, and support. We wish to thank Neil Epstein, Gert-Martin Greuel, Sean Howe, Jack Jeffries, Nawaj KC, Hong Duc Nguyen, and Eugenii Shustin for productive mathematical discussions and their inputs on the different ideas presented. We especially wish to thank Benjamin Baily for the main idea and for the technical details of Proposition~\ref{prop:det_other_direction}. The author was partially support by NSF grant DMS-2101800. 

\section{Determinacy over Complete Local Rings}\label{appendix_Det}

In this section we provide a proof of the finite determinacy theorem over any complete Noetherian local ring $(R, \mathfrak{m}, \kappa)$. These determinacy results are generalized versions of the results presented in Section 3.1 of~\cite{boubakri2009hypersurface} (which in turn are based on Section 2.2 of~\cite{greuel2007introduction} in the characteristic zero case and were expanded upon in Section 2 of~\cite{boubakri2012invariants} and in~\cite{greuel2019finite}).

\begin{definition}\label{def:determinacy}
Let $I \subset R[[\underline{x}]]$ be an ideal.
\begin{enumerate}
    \item We denote by $\mathcal{D}_I$ the group of $R-$automorphisms $\varphi$ of $R[[\underline{x}]]$ that preserve $I$ (i.e. $\varphi(I)=I$). Given $f,g \in I$ we say that $f \equivd g$ if there exists some $\varphi \in \mathcal{D}_I$ such that $\langle f \rangle = \langle \varphi(g) \rangle$. 
    \item  Let $f \in R[[\underline{x}]]$. We say that $f$ is \textbf{$k-$determined with respect to $I$} if for every $g \in R[[\underline{x}]]$ such that $f-g \in I^{k+1}$ we have that $f \equivd g$. We say that $f$ has \textbf{finite determinacy with respect to $I$} if $f$ is $k-$determined with respect to $I$ for some $k \geq 0$.
\end{enumerate} 
\end{definition}

\begin{example}\label{ex:A_1_det}
    \textup{Assume that $(R, \langle \pi\rangle, \kappa)$ is a DVR such that $\kappa$ is quadratically closed of characteristic $\neq 2$, and set $f=\pi^2 + x^2 \in R[[x]]$. We show directly that $f$ is $2-$determined with respect to $\mathfrak{M}=\langle x, \pi \rangle$. Given some $g\in R[[x]]$ such that $f - g \in \mathfrak{M}^3$, we can write $g = x^2+\pi^2 + h$ where $h \in \langle x,\pi \rangle^3$, and so $h=\alpha_1 x^3 + \alpha_2 x^2 \pi + \alpha_3 x\pi^2 + \alpha_4 \pi^3$ for some $\alpha_1, \alpha_2, \alpha_3, \alpha_4 \in R[[\underline{x}]]$. Therefore we have that}
    \begin{equation*}
        g = x^2+y^2 +h = (1+\alpha_3 x + \alpha_4\pi) (x^2(1+\alpha_1 x + \alpha_2\pi)(1+\alpha_3x + \alpha_4\pi)^{-1} + \pi^2).
    \end{equation*}
    \textup{Observe that both $1+\alpha_1 x + \alpha_2\pi$ and $1+\alpha_3x + \alpha_4\pi$ are units. Therefore, since $R$ is a DVR (and therefore Henselian) and since $\kappa$ is quadratically closed of characteristic $\neq 2$, there exists some unit $\beta \in R[[\underline{x}]]$ such that $\beta^2=(1+\alpha_1 x + \alpha_2\pi)(1+\alpha_3x + \alpha_4\pi)^{-1}$. So, if we define the automorphism $\varphi$ by setting $\varphi(x)=\beta x$ then we can conclude that $\varphi(f)=(1+\varphi(\alpha_3)\beta x + \varphi(\alpha_4)\pi) g$ with $1+\varphi(\alpha_3)\beta x+ \varphi(\alpha_4)\pi$ being a unit, and thus $f \equivk g$, as desired. }
\end{example}

\begin{remark}
    \textup{Determinacy over local rings that are algebras over an algebraically closed field of characteristic zero have been studied by Belitskii and Kerner in~\cite{belitskii2012matrices, belitskii2016finite, belitskii2019finite}, and in~\cite{boix2022pairs}, Boix et al. presented a "characteristic-free" approach to determinacy over local ring by looking at group actions on modules over local rings.  In addition, the study of the group $\mathcal{D}_I$ where $I$ is an ideal in the ring of power series over $\mathbb{C}$ was presented in~\cite{pellikaan1988finite}. }
\end{remark}

The main result of this section is the Determinacy theorem for $\equivd$ with respect to an ideal $I$:

\begin{theorem}[Finite Determinacy Theorem]\label{thm_App}
    Let $I \subset R[[\underline{x}]]$ be an ideal. If $f \in I^2 \subset R[[\underline{x}]]$ satisfies
    \begin{equation*}
        I^{k+2} \subset I \cdot \langle f \rangle + I^2 \cdot J(f),
    \end{equation*}
      then $f$ is $\left(2k-\ord_I\left(f\right)+2\right)-$determined with respect to $I$.
\end{theorem}

We start with a few technical lemmata:

\begin{lemma}\label{lem:app_computation}
    Let $f \in R[[\underline{x}]]$ and let $\underline{z}=(z_1, \dots, z_n)$ be a tuple of elements such that $z_i \in \mathfrak{M}^k$ for every $i$. Then we can write $f(x_1+z_1, \dots, x_n+z_n)=\sum_{\alpha} h_\alpha \underline{z}^\alpha$ for some $h_\alpha \in R[[\underline{x}]]$ with
    \begin{enumerate}
        \item $\ord(h_\alpha) \geq \ord(f) - |\alpha|$ for every $\alpha$, 
        \item $h_0=f$,
        \item $h_{e_j}=\partial_j(f)$ where $e_j$ is the $j-$th unit vector,
        \item $\ord(f(x_1+z_1, \dots, x_n+z_n) -  f(x_1, \dots, x_n)) \geq  \ord(f)+k-1$.
    \end{enumerate}
\end{lemma}

\begin{proof}
First, note that for every $\beta=(\beta_1, \dots, \beta_n) \in \mathbb{N}^n$ we have that 
\begin{equation*}
    \prod_{j=1}^n (x_j+z_j)^\beta =\sum_{\gamma_1=0}^{\beta_1} \cdots \sum_{\gamma_n=1}^{\beta_n} c_{\beta, \gamma}\underline{x}^{\beta-\gamma} \underline{z}^\gamma,
\end{equation*}

\noindent where $\gamma=(\gamma_1, \dots, \gamma_n)$ and $c_{\beta, \gamma} = \prod_{j=1}^n \binom{\beta_j}{\gamma_j}$. Thus, if we write $f(\underline{x})=\sum_{\beta} a_\beta \underline{x}^\beta$ then by setting $h_\alpha=\sum_{ \beta\geq \alpha} a_\beta c_{\beta, \alpha} \underline{x}^{\beta-\alpha}$, we have that 
\begin{equation*}
f(x_1+z_1, \dots, x_n+z_n) =     \sum_{\beta} a_\beta \cdot \sum_{\gamma_1=0}^{\beta_1} \cdots \sum_{\gamma_n=1}^{\beta_n} c_{\beta, \gamma}\underline{x}^{\beta-\gamma} \underline{z}^\gamma =\sum_\alpha h_\alpha \underline{z}^\alpha.  
\end{equation*}
\noindent Thus, since $\ord(f)=\min\{ |\beta| + \ord(a_\beta) \colon a_\beta \neq 0\}$ we have that $\ord(h_\alpha) = \min \{|\beta|-|\alpha| +\ord(a_\alpha) \colon \beta \geq \alpha\}\geq \ord(f)-|\alpha|$. Now, since $c_{\beta, 0}=1$ for every $\beta$, then we must have that 
\begin{equation*}
    h_0=\sum_{\beta} a_\beta  \underline{x}^{\beta} =f. 
\end{equation*}
\noindent In addition, if $e_j$ denotes the unit vector (i.e. the vector $(0, \dots, 0,1,0 \dots, 0)$ where $1$ appears only in the $j-$th component), then we have that $c_{\beta, e_j} = \beta_j$ and $\beta-e_j = (\beta_1, \dots, \beta_{j-1}, \beta_j -1, \beta_{j+1}, \dots, \beta_n)$, in addition to the fact that if $\alpha_j \leq e_j$ then $\alpha=0$. Therefore we must conclude that
\begin{equation*}
    h_{e_j} = \sum_{\beta} a_\beta \cdot \beta_j \cdot   \underline{x}^{\beta - e_j} = \partial_j(f).
\end{equation*}

Now, for the last item, note that $f(x_1+z_1, \dots, x_n+z_n)-f(x_1, \dots, x_n) = \sum_{|\alpha|>0} h_\alpha \underline{z}^\alpha$ (as $h_0=f=h_0 \underline{z}^0$). Since $z_i \in \mathfrak{M}^k$ for every $i$, then $\ord(\underline{z}^\alpha) \geq |\alpha|k$ and so we have that 
\begin{equation*}
    \ord(h_\alpha \underline{z}^\alpha) = \ord(h_\alpha) + \ord(\underline{z}^\alpha) \geq \ord(f) -|\alpha|+k|\alpha| \geq \ord(f)+k-1,
\end{equation*}
with $\ord(f(x_1+z_1, \dots, x_n+z_n)-f(x_1, \dots, x_n)) = \min\{ \ord(h_\alpha \underline{z}^\alpha) \colon |\alpha| \neq 0\}$. 
\end{proof}

\begin{lemma}\label{lem:app_fixpt}
    Let $\varphi \colon R[[\underline{x}]] \to R[[\underline{x}]]$ be an $R-$homomorpshim such that $\varphi(x_j) - x_j \in \mathfrak{M}^k$ for every $j$ (where $k \geq 2$). Then for every $h \in \mathfrak{M}^l$ we have that $\varphi(h)-h \in \mathfrak{M}^{l+k-1}$. 
\end{lemma}

\begin{proof}
    Since $\varphi(x_j) - x_j \in \mathfrak{M}^k$ then there exists some $a_j \in \mathfrak{M}^k$ such that $\varphi(x_j)=x_j+a_j$. Therefore, given $h \in R[[\underline{x}]]$, we have that $\varphi(h)=h(x_1+a_1, \dots, x_n + a_n)$, and so the result follows from the last item of Lemma~\ref{lem:app_computation}. 
\end{proof}

\begin{lemma}\label{lem:app_const}
Let $N,i \geq 1$ integers. Let $b_{i,0} \in \mathfrak{M}^{N+i-1}$ and $b_{i,j} \in \mathfrak{M}^{N+i}$ for every $j = 1, \dots, n$. Set $v_i = 1 +b_{i,0}$ and set $\varphi_i$ to be the $R-$automorphism of $R[[\underline{x}]]$ defined by $\varphi_i(x_j)=x_j+b_{i,j}$ for every $j = 1, \dots, n$. Set $\psi_i=\varphi_{i} \circ \varphi_{i-1} \circ \cdots \circ \varphi_1$ and set inductively $u_i = v_i \cdot \varphi(u_{i-1})$ with $u_0=1$. Then we have that:
\begin{itemize}
    \item For every $j$ there exists some $b_j \in \mathfrak{M}^{N+1}$ such that $\psi_i(x_j) \to x_j+b_j$ as $i \to \infty$. 
    \item The $R-$homomorphism $\varphi$ defined by $\varphi(x_j) = x_j+b_j$ for every $j$ defines an $R-$automorphism of $R[[\underline{x}]]$. 
    \item There exists some $b_0 \in \mathfrak{M}$ such that $u_i \to u=1+b_0$ as $i \to \infty$. 
\end{itemize}
\end{lemma}

\begin{proof}
    First, observe that $\psi_i = \varphi_i \circ \psi_{i-1}$. Therefore, by the definition of $\psi_i$, there exists some $c_{i,j} \in R[[\underline{x}]]$ such that $\psi_i(x_j)=x_j+c_{i,j}$ for every $i,j$ (with $c_{1,j}=b_{1,j}$ since $\psi_1=\varphi_1$). Note that $c_{i,j}$ satisfies the recursive relation $c_{i,j}=\varphi_i(c_{i-1,j})+b_{i,j}$, since 
    \begin{equation*}
        x_j + c_{i,j} = \psi_i(x_j) = \varphi_i(\psi_{i-1}(x_j))=\varphi_i(x_j+c_{i-1,j})= x_j + b_{i,j} +  \varphi_i(c_{i-1,j}),
    \end{equation*}
    \noindent and so we can conclude that
    \begin{equation*}
        \psi_i(x_j)-\psi_{i-1}(x_j) = c_{i,j} - c_{i-1,j} = b_{i,j} + \varphi_i(c_{i-1,j}) -c_{i-1,j}.
    \end{equation*}
    \noindent Now, since $b_{i,j} \in \mathfrak{M}^{N+i}$ with $N+i \geq 2$, then by applying Lemma~\ref{lem:app_fixpt} (with $h=\psi_i(x_j) \in \mathfrak{M}^1$ and $k=N+i$) to $\varphi_i$, we can conclude that $\varphi_i(\psi_{i-1}(x_j)) - \psi_{i-1}(x_j) \in \mathfrak{M}^{N+i}$, and since $\varphi_i(\psi_{i-1}(x_j))=\psi_i(x_j)$, we have that $\psi_i(x_j)-\psi_{i-1}(x_j) \in \mathfrak{M}^{N+i}$.\\
    
    Thus, for every $r \geq 1$ there exists some $l=\max\{r-N, 1\}$ such that for every $i_1 > i_2 > l$ we have that 
    \begin{equation*}
        \psi_{i_1}(x_j)-\psi_{i_2}(x_j) = \sum_{k=i_2+1}^{i_1} (\psi_k(x_j)- \psi_{k-1}(x_j)) \in \mathfrak{M}^{N+l-1} \subset \mathfrak{M}^r. 
    \end{equation*}
    \noindent Therefore we can conclude that the sequence $\{\psi_i(x_j)\}_i$ is  Cauchy sequence for every $j$, and since $R[[\underline{x}]]$ is a complete ring, we can conclude that for every $j$ there exists some $\psi(x_j)$  such that $\psi_i(x_j) \to \psi(x_j)$ as $i \to \infty$. We can also  conclude that for every $j$ there exists some $b_j$ such that $c_{i,j} \to b_j$ as $i \to \infty$. Therefore we have that $\psi(x_j)=x_j+b_j$. \\

    Now, since for every $i$ we have that $u_i=1+b_{i,0}$, then $\varphi_i(u_{i-1})-u_{i-1} = \varphi_i(b_{i,0})-b_{i,0}$, and so by applying Lemma~\ref{lem:app_fixpt} to $b_{i,0} \in \mathfrak{M}^{N+i-1}$ we can conclude that $ \varphi_i(b_{i,0})-b_{i,0} \in \mathfrak{M}^{N+i}$. Therefore we have that $u_i - u_{i-1} = (1+b_{i,0})\varphi_i(u_{i-1})-u_{i-1} \in \mathfrak{M}^{N+i-1}$. Therefore, as before, we get that $\{u_i\}_i$ is a Cauchy sequence that converges to some unit $u$. Note that by the definition of $u_i$, for every $i$ there exists some $d_i \in \mathfrak{M}$ such that $u_i=1+d_i$, and so $u=1+b_0$ for some $b_0 \in \mathfrak{M}$.  
\end{proof}

\begin{lemma}\label{lem:app_limit}
    Let $\{\varphi_i\}_{i=1} ^\infty$ be a sequence of $R-$automorpshims of $R[[\underline{x}]]$ and let $\varphi$ be an $R-$automorpshims of $R[[\underline{x}]]$ such that for every $j$ we have that $\varphi_i(x_j) \to \varphi(x_j)$ as $i \to \infty$. Let $\{u_i\}_{i=1} ^\infty$ be a sequence of units in $R[[\underline{x}]]$ that converges to some $u$. Then for every $g \in R[[\underline{x}]]$ we have that $u_i \cdot \varphi_i(g) \to u\cdot \varphi(g)$ as $i \to \infty$.
\end{lemma}

\begin{proof}
    Since $\varphi_i(x_j) \to \varphi(x_j)$ for every $j$ and $u_i \to u$, then there exists some $l$ and some $L$ such that $\varphi(x_j)-\varphi(x_j) \in \mathfrak{M}^l$ and $u-u_i \in \mathfrak{M}^l$ for every $i \geq L$. If we write $g=\sum_{\alpha} a_{\alpha}\underline{x}^{\alpha}$, then we can conclude that 
    \begin{equation*}
        \varphi(g)-\varphi_i(g)=\sum_{\alpha} a_{\alpha} (\varphi(\underline{x})^{\alpha} - \varphi_i(\underline{x})^{\alpha}) \in \mathfrak{M}^l,
    \end{equation*}
    \noindent and thus $u\varphi(g)-u_i\varphi(g)=u(\varphi(g)-\varphi_i(g)) + (u-u_i)\varphi_i(g) \in \mathfrak{M}^l$ for every $i \geq L$, and the result follows. 
\end{proof}

\begin{lemma}\label{lem:app_ord_der}
    Let $f \in R[[\underline{x}]]$ and let $I \subset R[[\underline{x}]]$ be an ideal. If $f \in I^k$ for some $k\geq 1$ then $\partial_j(f) \in I^{k-1}$ for every $j$. 
\end{lemma}

\begin{proof}
    We prove this result by induction. For $k=1$ the result is trivial as $I^{k-1}=R[[\underline{x}]]$. Now, assume that the result is true for some $k-1$ and we will show that it is true for $k$. Since $f \in I^k$ we can write $f = \sum_i a_i b_i$ where $a_i \in I$ and $b_i \in I^{k-1}$ for every $i$. Therefore we have that $\partial_j(f)=\sum_i (\partial_j(a_i)\cdot b_i + a_i \cdot \partial_j(b_i))$. Since $b_i \in I^{k-1}$ then by the induction step we have that $\partial_j(b_i) \in I^{k-2}$. Thus we can conclude that $\partial_j(a_i)\cdot b_i + a_i \cdot \partial_j(b_i) \in I^{k-1}$ for every $i$ and the result follows. 
    %We can uniquely write $f=\sum_{i=0}^\infty a_i x_j^i$ where $a_j \in R[[x_1, \dots, x_{j-1}, x_{j+1},x_n]]$. Denote by $r_i$ the order of $a_i$ and denote the order of $f$ by $r$. Then for every $j$ we must have that $r_i+i \geq r$. We have that $\partial_j(f)=\sum_{i=0}^\infty a_i \cdot i \cdot x_j^{i-1}$ (where we view $i$ as an element of $R$). Therefore, we have that $\ord(a_i \cdot i \cdot x_j^{i-1})=r_i+i-1+\ord(i) \geq r-1$. We can conclude that $a_i \cdot i \cdot x_j^{i-1} \in M^{r-1}$ for every $i$ and so $\partial_j(f) \in M^{r-1}$.
\end{proof}

\begin{remark}
 \textup{Note that we can rewrite Lemma~\ref{lem:app_ord_der} as $\ord_I(\partial_i(f)) \geq \ord_I(f)-1$ for every ideal $I \subset R[[\underline{x}]]$. If $R$ is a ring of either mixed or positive characteristic then the order of the derivative can substantially change and is unbounded from above. For example, if $R$ is of mixed characteristic $(0,p)$ and $f(x) = x^{p^r} \in R[[x]]$ for some $r>1$ then we have that $\text{ord}(\partial_x(f))=\text{ord}(f)+r-1$. }
\end{remark} 

\begin{lemma}\label{lem:D_I_preserves}
    If $\varphi$ is an $R-$automorphism of $R[[\underline{x}]]$ and $I \subset R[[\underline{x}]]$ is an ideal such that $\varphi(I) \subset I$ then $\varphi \in \mathcal{D}_I$. 
\end{lemma}

\begin{proof}

    Assume towards contradicition that $\varphi(I) \subsetneq I$. Then since $\varphi$ is an isomorphism we have that $I \subsetneq \varphi^{-1}(I)$. Therefore, if we define inductively the sequence of ideals $I_1=I$ and $I_{k+1} = \varphi^{-1}(I_k)$ we can conclude that $I_k \subset I_{k+1}$ for every $k$. This is true since if $k$ is the smallest integer such that $I_{k}=I_{k+1}$ then, since $\varphi$ is an isomorphism, we can conclude that $I_{k-1}=\varphi(I_{k})=\varphi(I_{k+1})=I_k$, which contradicts the minimality of $k$. Also note that $I_k \neq R[[\underline{x}]]$ since $\varphi$ is an isomorphism.  Therefore we can conclude that $\{I_k\}_{k=1}^\infty$ forms a strictly increasing sequence of non trivial ideals, which contradicts the Noetherianity of $R[[\underline{x}]]$.  
\end{proof}

%Since $R[[\underline{x}]]$ then $I$ is a finitely generated module, and therefore $\frac{I}{\mathfrak{M}I}$ is a finite dimensional vector space over $\frac{R[[\underline{x}]]}{\mathfrak{M}}$. Since $\varphi(I) \subset I$ and since $\varphi$ is an automorphism, we can conclude that the induced map $\overline{\varphi} \colon \frac{I}{\mathfrak{M}I} \to \frac{I}{\mathfrak{M}I}$ is injective. Therefore, $\overline{\varphi}$ must be a bijection with $\overline{\varphi}(\frac{I}{\mathfrak{M}I})=\frac{I}{\mathfrak{M}I}$, and so by Nakayama's lemma we can conclude that $\varphi$ is surjective on $I$, i.e $\varphi(I)=I$, as desired

\begin{remark}
\textup{The assumption that $R[[\underline{x}]]$ in Noetherian in Lemma~\ref{lem:D_I_preserves}. For example, let $k[[x_1, x_2, \dots ]]$ be the ring of power series with countably many variables over a field $k$, let $\varphi$ be the automorphism of $k[[x_1, x_2, \dots ]]$ defined by $\varphi(x_i)=x_i+x_{i+1}$ for every $i$, and let $I=\langle x_2, x_3, \dots \rangle$. Then we have that $\varphi(I) \subsetneq I$ since $x_2 \notin \varphi(I)$. }
\end{remark}

\begin{lemma}\label{lem:W[[x]]_good_prop}
\begin{enumerate}
    \item \textup{(Implicit Function Theorem)}: Given a sequence of power series $f_1, \dots, f_m \in R[[\underline{x},y_1, \dots y_m]]$ such that $f_i(\underline{0})=0$ for every $i$ and 
    \begin{equation*}
    \det \left(
        \begin{bmatrix}
            \frac{\partial f_1}{\partial y_1}(\underline{0}) & \cdots & \frac{\partial f_1}{\partial y_m}(\underline{0})\\
            \vdots & & \vdots \\
            \frac{\partial f_m}{\partial y_1}(\underline{0}) & \cdots & \frac{\partial f_m}{\partial y_m}(\underline{0})
        \end{bmatrix}
        \right) \notin \mathfrak{m},
    \end{equation*}
    then we can find a sequence of power series $g_1, \dots, g_m \in \mathfrak{M} \subset R[[\underline{x}]]$ for which we have that $f_i\left(x_1, \dots, x_n, g_1, \dots g_m\right)=0$ for every $i$. 
    \item \textup{(Inverse Function Theorem)}: Let $f_1, \dots, f_n \in \mathfrak{M} \subset R[[\underline{x}]]$. Then the  Jacobian matrix $Df=\left[\frac{\partial f_i}{\partial x_j}\left(0\right)\right]_{i,j} \in \Mat_n\left(R\right)$ is an invertible matrix if and only if the unique $R-$homomorphism $\varphi \colon R[[\underline{x}]] \to R[[\underline{x}]]$ defined by $\varphi\left(x_i\right)=f_i$ defines an $R-$automorphism. 
    %\item (\textbf{Approximation Theorem}): Let $\left(f_1, \dots, f_N\right) \in \Vx^N$ and let $\{Y_c\}_{c=1}^\infty \subset \Vx^N$. Assume that there exists a monotonically increasing function $\beta \colon \mathbb{N} \to \mathbb{N}$ such that $f_i\left(Y_c\right) \in \maxV^{\beta\left(c\right)}$ for every $i$ and for every $c$. Then, there exists some $Y \in \Vx^N$ such that $Y - Y_c \in \maxV^c$ and $f_i\left(Y\right)=0$ for every $i$. 
\end{enumerate}
\end{lemma}

\begin{proof}
    \begin{enumerate}
        \item Since $R[[\underline{x}]]$ is a complete local ring, then by Lemma 10.153.9. in~\cite{stacks-project} (or by corollary 1.9 in~\cite{leuschke2012cohen}) it must be Henselian. Thus, given $f_1, \dots, f_m \in R[[\underline{x},y_1, \dots, y_m ]]=R[[\underline{x}]][[y_1, \dots, y_m]]$ such that satisfy the condition above, if we choose to define $\overline{f_i} = f_i \mod \mathfrak{M} \in \kappa[[y_1, \dots, y_m]]$ for every $i$ we must have that $\overline{f_i}\in \langle y_1, \dots, y_m \rangle \subset \kappa[[\underline{y}]]$ and that the Jacobian matrix of $(\overline{f_1}, \dots, \overline{f_m})$ is an invertible matrix over $\kappa$. Therefore, by the Implicit function theorem over $\kappa$ and by the Multi-dimensional version of Hensel's lemma (see, for example, Section 4 in~\cite{kuhlmann2000valuation}, Section 7 in~\cite{ziegler1978model}, or Section 1.2 of Chapter I in~\cite{greuel2007introduction}), we can find some unique $\overline{g_1}, \dots, \overline{g_m}$ such that $\overline{f_i}\left(x_1, \dots, x_n, \overline{g_1}, \dots, \overline{g_m}\right)=0$ for every $i$ that uniquely lift to some $g_1, \dots, g_m \in R[[\underline{x},y_1, \dots y_m]]$ such that $f_i\left(x_1, \dots, x_n, g_1, \dots, g_m\right)=0$ for every $i$, as desired.
        \item If $\varphi$ is an automorphism then there exists some $\psi$ such that $\varphi \circ \psi$ is the identity, and so the Jacobian matrix of $\varphi$ would be invertible by the chain rule. For the other direction, in order to construct for $\varphi$ an inverse, it is enough to find for every $i$ some $h_i$ such that $f_i(h_1, \dots, h_n)=x_i$ for every $i$, and this would follow from the first item. 
    \end{enumerate}
\end{proof}

\begin{lemma}[Lifting Lemma]\label{lem:lifting}
     Let $I, J \subset  \mathfrak{M}$ be two ideals and assume that $\varphi \colon \frac{R[[\underline{x}]]}{I} \to \frac{R[[\underline{x}]]}{J}$ is some (local) $R-$homomorphism. Then there exists some $R-$homomorphism $\Phi \colon R[[\underline{x}]] \to R[[\underline{x}]]$ such that $\Phi(I) \subset J$ and the following diagram commutes:
    \begin{equation*}
        \begin{tikzcd}
    R[[\underline{x}]] \arrow[r, "\Phi"] \arrow[d]
    & R[[\underline{x}]] \arrow[d, ] \\
    \frac{R[[\underline{x}]]}{I} \arrow[r,  "\varphi" ]
    & |[]| \frac{R[[\underline{x}]]}{J}.
    \end{tikzcd}
    \end{equation*}
    In addition, if we have that $\varphi$ is an $R-$isomorphism then in fact we can choose the $R-$homomorphism $\Phi$ to be an $R-$isomorphism as well with $\Phi(I)=J$. 
\end{lemma}

\begin{proof}
    Denote by $\overline{x_i}$ to be the image of $x_i$ in $\frac{R[[\underline{x}]]}{I}$, and choose some $f_i \in R[[\underline{x}]]$ to be any element in the preimage of $\varphi(\overline{x_i})$ under the projection $R[[\underline{x}]] \to \frac{R[[\underline{x}]]}{J}$. Then we can define $\Phi$ by setting $\Phi(x_i) = f_i$ for every $i$.\\
    
    Now, assume in addition that $\varphi$ is an automorphism. In this case we need to perform a specific choice of $f_i$. Let $g_1, \dots, g_{n+1} \in R[[\underline{x}]]$ induce a basis for $\frac{\mathfrak{M}}{\mathfrak{M}^2}$ over $\kappa$ such that $g_1, \dots, g_l$ for a basis for $\frac{\mathfrak{M}}{I+\mathfrak{M}^2}$ over $\kappa$ and $g_{l+1}, \dots, g_{n+1}\in I$. Note that this is possible since we have a short exact sequence 
    \begin{equation*}
        0 \to \frac{I+\mathfrak{M}^2}{\mathfrak{M}^2} \to \frac{\mathfrak{M}}{\mathfrak{M}^2} \to \frac{\mathfrak{M}}{I+\mathfrak{M}^2} \to 0.
    \end{equation*}
    Note that by Lemma~\ref{lem:W[[x]]_good_prop} we have that the $R-$homomorphism defined by $\Psi(x_i)=g_i$ is an automorphism. \\ 

    Denote $\overline{g_i}=g_i \mod I$ and $\overline{h}_i=\varphi(\overline{g_i})$. Note that $\overline{h}_i=0$ and $\overline{g}_i=0$ for every $i > l$. In addition, $\overline{g_1}, \dots, \overline{g_l}$ form a basis for $\frac{\mathfrak{M}}{I+\mathfrak{M}^2}$ over $\kappa$. Since $\varphi$ is an isomorphism, then we have that $\overline{h_1}, \dots, \overline{h_l}$ form a basis for $\frac{\mathfrak{M}}{J+\mathfrak{M}^2}$ over $\kappa$. Lift $\overline{h_1}, \dots, \overline{h_l}$ from $\frac{R[[\underline{x}]]}{J}$ to some $h_1, \dots, h_l \in R[[\underline{x}]]$, and choose some $h_{l+1}, \dots h_{n+1}$ such that $h_1, \dots, h_{n+1}$ form a basis for $\frac{\mathfrak{M}}{\mathfrak{M}^2}$ over $\kappa$. \\

    Note that by Lemma~\ref{lem:W[[x]]_good_prop} we have that the $R-$homomorphism defined by $\Theta(x_i)=h_i$ is an automorphism. Therefore, if we set $\Phi=\Theta \circ \Psi^{-1}$, then $\Phi$ is an automorphism that satisfies $\Phi(g_i)=h_i$, and so $\Phi$ is a lift of $\varphi$ that is an isomorphism as well. The fact that $\varphi(I)=J$ follows from a similar argument to that of Lemma~\ref{lem:D_I_preserves} (noting that since $\varphi$ is an isomorphism then $\dim\left(\frac{R[[\underline{x}]]}{I}\right) = \dim\left(\frac{R[[\underline{x}]]}{J}\right)$, and so $I$ and $J$ have the same height). 
\end{proof}

\begin{remark}
    \textup{Note that Lemma~\ref{lem:W[[x]]_good_prop} and Lemma~\ref{lem:lifting} give us analogues to some classical results in the field case. The first item of Lemma~\ref{lem:W[[x]]_good_prop} is a generalized version of~\cite{tougeron1968ideaux} (for more details, see Section 3.1 in~\cite{hauser2013artin}, Section 3 in~\cite{belitskii2016strong}, or~\cite{kerner2024orbits}) and the second item is a version of Corollary 3.3.7 in~\cite{de2013local}. Lemma~\ref{lem:lifting} is a generalized version of Lemma 1.14 and Lemma 1.23 in~\cite{greuel2007introduction} and of Lemma 1.3.5. in~\cite{boubakri2009hypersurface}.}
\end{remark}

\begin{lemma}\label{lem:equi_cont}
    Let $f$  and let $\{f_i\}_{i=1}^\infty \subset R[[\underline{x}]]$ be a sequence such that $f_i \to f$ (in the $\mathfrak{M}-$adic topology) with $f_i \equivk g$ for every $i$ for some $g$, then $f \equivk g$.
\end{lemma}

\begin{proof}
    For every $i$  there exists an automorphism $\varphi_i$ of $R[[\underline{x}]]$ and a unit $u_i$ such that $f_i=u_ig \circ  \varphi_i$. Therefore we have that for every $N>0$ there exists some $\beta(N)$ such that $f - u_i g \circ \varphi_i \in \mathfrak{M}^N$ for every $i> \beta(N)$. Therefore, if we define $F \in R[[x_1, \dots, x_n, w_1, \dots, w_n,t]]$ by
    \begin{equation*}
        F(x_1, \dots, x_n, w_1, \dots, w_n, t)=f (x_1, \dots, x_n)-g(w_1, \dots, w_n),
    \end{equation*}
     then we have that $F(\underline{x}, \varphi_i(\underline{x}),u_i)=f-u_i g \circ \varphi_i \to 0$.  So by the Artin's strong approximation theorem over local rings (see, for example, Theorem 3.3. in~\cite{rond2015artin}, which in turn cites~\cite{schappacher1980some, schappacher1983inegalite}) we have that there exists some automorphism $\varphi$ and some unit $u$ such that $u_i\varphi_i(\underline{x})-u\varphi(\underline{x}) \in \mathfrak{M}^i$ and $F(\underline{x}, \varphi(\underline{x}),u)=0$, which gives us that $f=ug(\varphi)$. 
\end{proof}

\begin{remark}
\textup{Artin's approximation theorem that we use in the proof of Lemma~\ref{lem:equi_cont}, which is based upon Artin's original result in~\cite{artin1969algebraic}, is an essential result in the theory of isolated singularities over fields. For more information and details, see~\cite{belitskii2020approximation, hauser2013artin}.  }
\end{remark}

\begin{lemma}\label{lem:aut_der}
    Let $f \in R[[\underline{x}]]$, let $\varphi$ be an $R-$automorphism of $R[[\underline{x}]]$, and let $u \in R[[\underline{x}]]$ be a unit. Then $J(uf)+\langle uf \rangle = J(f)+\langle f \rangle$ and $\varphi(J(f))=J(\varphi(f))$. 
\end{lemma}

\begin{proof}
    First, by the product rule we have that $\partial_i(uf)=f\partial_i(u)+u\partial_i(f)$, which directly gives us that $J(uf)+\langle uf \rangle = J(f)+\langle f \rangle$. Now, given an $R-$automorpshim $\varphi$, then by the chain rule we have an equality of vectors
    \begin{equation*}
        (\partial_1(\varphi(f)), \dots, \partial_n(\varphi(f)))= (\partial_1(f), \dots, \partial_n(f)) \cdot D\varphi,
    \end{equation*}
    where $D\varphi$ is the Jacobian matrix of $f$ (as in Lemma~\ref{lem:W[[x]]_good_prop}). Therefore we can conclude that $J(\varphi(f)) \subset \varphi(J(f))$. By Lemma~\ref{lem:W[[x]]_good_prop} we have that $D\varphi$ is invertible as $\varphi$ is an automorphism. Therefore we have an equality of vectors $ (\partial_1(\varphi(f)), \dots, \partial_n(\varphi(f))) \cdot D\varphi^{-1}= (\partial_1(f), \dots, \partial_n(f))$ from which we can conclude that $\varphi(J(f)) \subset J(\varphi(f))$, and the result follows. 
\end{proof}

\begin{remark}
    \textup{The proof of Lemma~\ref{lem:aut_der} is similar to that of Lemma 1.2.7. in~\cite{boubakri2009hypersurface} and to that of Lemma 2.10. in~\cite{greuel2007introduction}. }
\end{remark}

\begin{lemma}\label{lem:D_Iclosed}
    Let $I \subset R[[\underline{x}]]$ be an ideal. Then $\mathcal{D}_I$ is closed, i.e. if $\{\varphi_j\}_{j=1}^\infty \subset \mathcal{D}_I$ such that for every $0 \leq i \leq n$ we have that $\lim_{j \to \infty} \varphi_j(x_i)=a_i$, then the $R-$automorphism $\varphi$ of $R[[\underline{x}]]$ defined by $\varphi(x_i)=a_i$ for every $i$ is in $\mathcal{D}_I$. 
\end{lemma}

\begin{proof}
    First, note that $\varphi$ is an automorphism since by Lemma~\ref{lem:W[[x]]_good_prop} we have that the Jacobian matrix $D\varphi_j$ is invertible for every $j$, and since we have that $D\varphi_j \to D\varphi$ as $j \to \infty$ (in the $\mathfrak{M}-$adic topology), we have that $D\varphi$ is invertible as well. Now, if $I=\langle a_1, \dots, a_s \rangle$ then by Lemma~\ref{lem:D_I_preserves} the sequence $\varphi_j(a_1), \dots, \varphi_j(a_s)$ is a generating sequence for $I$ for every $j$. Therefore, since $I$ is closed in the $\mathfrak{M}-$adic topology we must have that $\varphi(a_1), \dots, \varphi(a_s) \in I$. Since $\varphi$ is an automorphism, then again by Lemma~\ref{lem:D_I_preserves} we must in fact have that $\varphi(a_1), \dots, \varphi(a_s)$ is a generating sequence of $I$, and so $\varphi \in \mathcal{D}_I$. 
\end{proof}

\begin{remark}
    \textup{We can conclude from Lemma~\ref{lem:D_Iclosed} a version of Lemma~\ref{lem:equi_cont} for $\equivd$. Specifically, given $\{f_i\}_{i=1}^\infty \subset I$ such that $f_i \to f$ (in the $\mathfrak{M}-$adic topology) and $f_i \equivd g$ for every $i$ then $f \equivd g$.    }
\end{remark}

We are now ready to prove Theorem~\ref{thm_App}:

\begin{proof}[Proof of Theorem~\ref{thm_App}]
    Denote $r=\ord_I(f)$ . Then by Lemma~\ref{lem:app_ord_der} we have that $\partial_j(f) \in I^{r-1}$. Thus, we get that 
    \begin{equation*}
        I^{k+2} \subset I \cdot \langle f \rangle + I^2 \cdot J(f) \subset I^{r+1}. 
    \end{equation*}
    Set $N=2k-r+2$, and we know that $N \geq k+1$. Let $g \in R[[\underline{x}]]$ such that $f-g \in I^{N+1}$. We show that $f \equivd g$, i.e., there exists some unit $u$ and some $R-$automorphism $\psi \in \mathcal{D}_I$ of $R[[\underline{x}]]$ such that $g=u\cdot \psi(f)$. We construct a sequence of units $\{u_i\}_{i=1}^\infty$ such that $u_i \to u$ as $i \to \infty$ and a sequence of $R-$automorphisms $\{\psi_i\}_{i=1}^\infty \subset \mathcal{D}_I$ of $R[[\underline{x}]]$ such that $\psi_i(f) \to \psi(f)$ as $i \to \infty$ for some $R-$automorphism $\psi$ of $R[[\underline{x}]]$ with $g - u_i \psi_i(f) \in \mathfrak{M}^{N+1+i}$ for every $i$. By  Lemma~\ref{lem:app_limit} we would get that $g = u \cdot \psi(f)$ and by Lemma~\ref{lem:D_Iclosed} we have that $\psi \in \mathcal{D}_I$. Our construction relies on the specific settings of Lemma~\ref{lem:app_const}.\\
    
    Note that by the condition $I^{k+2} \subset I \cdot \langle f \rangle + I^2 \cdot J(f) \subset I^{r+1}$ we must have that 
    \begin{equation*}
        g-f \in I^{N+1} \subset I^L \cdot \langle f \rangle + I^{L+1} \cdot J(f),
    \end{equation*}
    \noindent where $L=N-k \geq 1$. Observe that $2L+r=N+2$.  Thus, there exists some $b_{1,0} \in I^L \subset \mathfrak{M}^L$ and $b_{1,j} \in I^{L+1} \subset \mathfrak{M}^{L+1}$ for every $j$ such that 
    \begin{equation*}
        g-f = b_{1,0} \cdot f + \sum_{j=1}^n b_{1,j} \cdot \partial_j (f).
    \end{equation*}
    \noindent We set $v_1=1+b_{1,0}$ and set $\varphi_1$ to be the $R-$automorphism of $R[[\underline{x}]]$ by $\varphi_1(x_j)=x_j+b_{1,j}$. Observe that $\varphi_1 \in \mathcal{D}_I$. We show that $g-v_1\cdot \varphi_1(f) \in \mathfrak{M}^{N+2}$. By applying Lemma~\ref{lem:app_computation} where $z_j=b_{1,j}$, we get that $\varphi_1(f)=f+\sum_j \partial_j(f) b_{1,j} +h$ with $h=\sum_{|\alpha|>1} h_\alpha b_{1,1}^{\alpha_1} \cdots b_{1,n}^{\alpha_n} \in \mathfrak{M}^{N+2}$, since 
    \begin{equation*}
        \ord(h_\alpha b_{1,1}^{\alpha_1} \cdots b_{1,n}^{\alpha_n}) \geq \ord(h_\alpha) + \sum_j \alpha_j \cdot \ord(b_{1,j}) \geq r-|\alpha|+(L+1)|\alpha|\geq r+2L. 
    \end{equation*}
    \noindent Since we have that $\ord(b_{1,0}b_{1,j}\partial_j(f)) \geq L+(L+1)+(r-1) = r+2L$ and since $g-f = b_{1,0} \cdot f + \sum_{j=1}^n b_{1,j} \cdot \partial_j (f)$, we must have that
    \begin{equation*}
        g-v_1\cdot \varphi_1(f)=g-(1+b_{1,0})\cdot(f+\sum_j\partial_j(f) b_{1,j} + h ) = -\sum_{j} b_{1,0}b_{1,j}\partial_j(f) - (1+b_{1,0})h, 
    \end{equation*}
    \noindent which must be contained in $\mathfrak{M}^{N+2}$.  Now, we can repeat this process inductively (using the notations of Lemma~\ref{lem:app_const}) to get collection of sequences $\{ \{b_{i,j}\}_{i=1}^\infty \colon j=1, \dots, n\}$ such that $b_{i,0} \in \mathfrak{M}^{N+i-1}$ and $b_{i,j} \in \mathfrak{M}^{N+i}$ for every $j=1, \dots, n$. Therefore by setting $\varphi_i(x_j) = x_j+b_{i,j}$ for every $i,j$, we get that $g-v_i\cdot \varphi_i(f) \in \mathfrak{M}^{N+1+i}$ and that $\varphi_i \in \mathcal{D}_I$, as desired.  
\end{proof}

\begin{corollary}\label{cor:sqrtI}
    If $I \subset \sqrt{\langle f \rangle + J(f)}$ then $f$ has finite determinacy with respect to $I$. 
\end{corollary}

\begin{proof}
    If $I \subset \sqrt{\langle f \rangle + J(f)}$ then by there exists some $k$ such that $I^k \subset \langle f \rangle + J(f)$, and so the result follows from Theorem~\ref{thm_App}.
\end{proof}

\begin{remark}
    \textup{Theorem~\ref{thm_App} generalizes Proposition 2.2 in~\cite{cluckers2022log}. In fact, following the notations of their paper, Theorem~\ref{thm_App} in fact gives us that if $f$ is $k-$determined with respect to $I$ and $\tilde{\alpha}(f) \geq \textup{lct}(\langle f \rangle+ I^{k})$.   }
\end{remark}

We now turn to proving a variant of the reverse direction of Corollary~\ref{cor:sqrtI}. In order to do so, we recall a version of the local Bertini theorem for complete local rings, as proven by Fenner in~\cite{flenner1977satze} for equi-characteristic rings and by Trivedi in~\cite{vijaylaxmi1994local} for mixed-characteristic rings. For more information, see Lemma 3.10 in~\cite{baily2026extremal} and~\cite{ochiai2015bertini}.

\begin{proposition}[Flenner-Trivedi]\label{prop:bertini-schmertini}
    Let $A$ be a  local Noetherian ring and let $I \subset A$ be an ideal. Then there exists some $x \in I$ such that $x \notin \mathfrak{p}^{(2)}= \mathfrak{p}^2 R_{\mathfrak{p}} \cap R$ for every $I \not\subset \mathfrak{p}$,  
\end{proposition}

\begin{remark}\label{rem:symbolicpower}
\begin{enumerate}
    \item \textup{Recall that given some $\mathfrak{p} \subset A$ prime ideal and some $f \in A$ (where $A$ local) we have that $\left( \frac{A}{\langle f \rangle} \right)_\mathfrak{p}$ is regular if and only if $f \notin \mathfrak{p}^{(2)}$. In addition, over $R[[\underline{x}]]$ we have that if $f \in \mathfrak{p}^{(2)}$ then $J(f) \subset \mathfrak{p}$.  For more information on symbolic power of an ideal, see~\cite{dao2015symbolic}. }
    \item \textup{Proposition~\ref{prop:bertini-schmertini} is not true if we look for an element $x \in I$ such that $x \notin \mathfrak{p}^{(2)}$ for every $I \subset \mathfrak{p}$. For example, if $I = \mathfrak{q}^2$ for some prime ideal $\mathfrak{q}$ then $I$ must be contained in $\mathfrak{p}^{(2)}$ for every prime ideal $I \subset \mathfrak{p}$, and so given some $x \in I$, then $x \in \mathfrak{p}^{(2)}$ for every prime ideal $I \subset \mathfrak{p}$.}
    \item \textup{Recall that the singular locus of $\frac{R[[\underline{x}]]}{\langle f \rangle}$, denoted $\textup{Sing}(V(f))$, is the set of prime ideals $\mathfrak{p} \subset R[[\underline{x}]]$ for which $\left( \frac{R[[\underline{x}]]}{\langle f \rangle}\right)_{\mathfrak{p}}$ is a regular local ring. By the previous item, $\textup{Sing}(V(f))$ is exactly the set of all prime ideals for which $f \in \mathfrak{p}^{(2)}$ and $\textup{Sing}(V(f)) \subset V(J(f) + \langle f \rangle)$. In addition, from Lemma~\ref{lem:aut_der} we can conclude that if $\varphi(f)=g$ for some $R-$automorphism of $R[[\underline{x}]]$ then $\varphi^*(\textup{Sing}(V(f))=\textup{Sing}(V(g))$. For more information on the singular locus of a ring, see Sections 15.47 and 15.48 in~\cite{stacks-project}.}
\end{enumerate}
\end{remark}

\begin{proposition}\label{prop:det_other_direction}
    Let $I \subset R[[\underline{x}]]$ be an ideal and let $f \in R[[\underline{x}]]$ be finite determined with respect to $I$. Then $\textup{Sing}(V(f)) \subset V(I + \langle f \rangle + J(f))$.
\end{proposition}

\begin{proof}
    Assume that $f \in R[[\underline{x}]]$ is $k-$determined with respect to $I$ (up to increasing $k$ we can assume that $k>0$). By Proposition~\ref{prop:bertini-schmertini} there exists some $g \in I^{k+1}$ such that $f+g \notin \mathfrak{p}^{(2)}$ for every prime ideal $\langle f \rangle + I^{k+1} \not\subset \mathfrak{p}$. \\
    
    We show that $\textup{Sing}(V(f+g)) \subset V(I + J(f) + \langle f \rangle)$. Given some prime ideal $\mathfrak{p}$ such that $I + J(f) + \langle f \rangle \not\subset \mathfrak{p}$, we have two options: First, if $I \not\subset \mathfrak{p}$ then $\langle f \rangle +I \not\subset \mathfrak{p}$ and so by the choice of $g \in I$ we must have that $f+g \notin \mathfrak{p}^{(2)}$. Second, if $I \subset \mathfrak{p}$ then as $I + J(f) + \langle f \rangle \not\subset \mathfrak{p}$ we must have that $J(f)+\langle f \rangle \not\subset \mathfrak{p}$ and so by the first item of Remark~\ref{rem:symbolicpower} we have that $f \notin \mathfrak{p}^{(2)}$ and so $f+g \notin \mathfrak{p}^{(2)}$. Therefore, from the third item of Remark~\ref{rem:symbolicpower}, we can conclude that if $\mathfrak{p} \in \textup{Sing}(V(f+g))$ then $I + J(f) + \langle f \rangle \subset \mathfrak{p}$.\\

    Since $f \in R[[\underline{x}]]$ is $k-$determined with respect to $I$ then $f+g \equivd f$ and so there exists some $\varphi \in \mathcal{D}_I$ such that $\langle \varphi(f) \rangle  = \langle f+g \rangle$. By Lemma~\ref{lem:aut_der} we have that $\varphi(I + J(f) + \langle f \rangle) = I + J(f+g) + \langle f+g \rangle$ (as $\varphi \in \mathcal{D}_I$). Since $g \in I^{k+1}$ then from Lemma~\ref{lem:app_ord_der} we have that $\partial_i(g) \in I^{k} \subset I$ and therefore $I + J(f) + \langle f \rangle = I + J(f+g) + \langle f+g \rangle$. Thus, by the third item of Remark~\ref{rem:symbolicpower}, we can conclude that $\textup{Sing}(V(f)) \subset V(I + J(f) + \langle f \rangle)$ as desired. 
\end{proof}

\begin{theorem}\label{thm:det_field_case}
    Assume that $\textup{Sing}(V(f)) = V(J(f) + \langle f \rangle)$. Then $f$ has finite determinacy with respect to $I$ if and only if  $I \subset \sqrt{\langle f \rangle + J(f)}$. 
\end{theorem}

\begin{proof}
    This follows directly from Corollary~\ref{cor:sqrtI} and Proposition~\ref{prop:det_other_direction}, recalling that $V(I + J(f) + \langle f \rangle) = V(I) \cap V(J(f)+\langle f \rangle)$. 
\end{proof}

\begin{remark}
\begin{enumerate}
    \item \textup{The assumption $\textup{Sing}(V(f)) = V(J(f) + \langle f \rangle)$ in Theorem~\ref{thm:det_field_case} is true for every $f \in R[[\underline{x}]]$ when $R$ is equi-characteristic (i.e. contains a field) and this follows from the classical Jacobian criterion (see, for example, Theorem 6.28 in~\cite{gortz2010algebraic}). In particular, in the equicharacteristic case, we can view Theorem~\ref{thm:det_field_case} as a generalization of Theorem 2.8 in~\cite{boubakri2012invariants} to hypersurfaces with arbitrary singularities (noting that their proof relies on a careful analysis of the action of the automorphism group). Yet, as we see in Theorem~\ref{thm:J_smooth} and in Proposition~\ref{prop:ramified_J}, we might have that $\textup{Sing}(V(f)) \subsetneq V(J(f) + \langle f \rangle)$ if $R$ is mixed characteristic.  }
    \item \textup{As in Section 3.1 of~\cite{boubakri2009hypersurface}, we can show an analogous results to those of Theorem~\ref{thm_App} and of Corollary~\ref{cor:sqrtI} for "right equivalence with respect to $I$". That is, we can say that $f \in R[[\underline{x}]]$ is $k-$right determined with respect to $I$ if for every $g \in R[[\underline{x}]]$ with $f-g \in I^{k+1}$ there exists some $\varphi \in \mathcal{D}_I$ such that $\varphi(f)=g$, and we say that $f$ is finite right determined if $f$ is $k-$right determined for some $k$. Therefore, a similar proof to that of Theorem~\ref{thm_App} tells us that if $I^{k+1} \subset I \cdot J(f)$ then $f$ is $(2k-\ord_I(f)+2)-$right determined with respect to $I$, and a similar proof to that of Corollary~\ref{cor:sqrtI} would tell us that if $I \subset \sqrt{J(f)}$ then $f$ is finite right determined.  }
\end{enumerate}
\end{remark}

Inspired by Corollary~\ref{cor:sqrtI}, we define an invariant that quantifies the determinacy of an element $f \in R[[\underline{x}]]$ with respect to an ideal $I$.

\begin{definition}
    Let $I \subset R[[\underline{x}]]$ and let $f \in I^2 \subset  R[[\underline{x}]]$. The \textbf{Jacobian number of $f$ with respect to $I$}  is defined to be 
    \begin{equation*}
    j_I(f)=\length_{R[[\underline{x}]]} \left(\frac{I}{\langle f \rangle + J(f)}\right). 
    \end{equation*}
\end{definition}

\begin{remark}
\begin{enumerate}
    \item \textup{Note that if $f \in I^2$ then by Lemma~\ref{lem:app_ord_der} we have that $\langle f \rangle + J(f) \subset I$, and so $j_I(f)$ is well defined. In addition, $j_I(f)$ if finite if and only if $I_\mathfrak{p} = (\langle f \rangle +J(f))_\mathfrak{p}$ for every prime ideal $\mathfrak{p} \neq \mathfrak{M}$. }
    \item \textup{If we replace $R$ by a field $k$ and select $I=\langle x_1, \dots, x_n\rangle$, then we have that $j(f)$ would equal to $\tau(f)-1$, where $\tau(f)=\dim_k(\frac{k[[\underline{x}]]}{\langle f\rangle + J(f)})$ is the Tjurina number of $f$ over $k$. In fact, we see a similar property in Propostion~\ref{lem:reduc_a}. For more information on the Tjurina number of fields and its connection to the topology and algebra of hypersurface singularities, see~\cite{greuel2007introduction, milnor2016singular, tjurina1969locally, varchenko1985singularities}.}
    \item \textup{Note that the definition and the notation of the Jacobian number $j_I(f)$ is inspired by the work of Siersma in~\cite{siersma1983isolated}, Pellikaan in~\cite{pellikaan1990deformations}, and Svoray in~\cite{svoray2024invariants}. }
\end{enumerate}
\end{remark}

\begin{proposition}\label{prop:jacobian_prop}
    Let $I \subset R[[\underline{x}]]$ and let $f \in R[[\underline{x}]]$ with $j_I(f)<\infty$. Then:
    \begin{enumerate}
        \item If $f \equivd g$ then $j_I(f)=j_I(g)$.
        \item $f$ is $(2 \cdot j_I(f) -\ord_I(f)+2)$-determined with respect to $I$. 
        \item If $\frac{R[[\underline{x}]]}{I}$ is Cohen-Macaulay of dimension $\dim(R) + n- s$ then  $j_I(f) \geq \binom{s+\ord_I(f)-2}{s+1}$.
    \end{enumerate}
\end{proposition}

\begin{proof}
    The first item follows from Lemma~\ref{lem:aut_der}. For the second item, assume that $j_I(f) < \infty$, then it is enough to show that $I^{j_I(f)} \subset \langle f \rangle + J(f)$, as then the finite determinacy would follow from Theorem~\ref{thm_App}. We have sequence of strict inclusions of modules
    \begin{equation*}
        \frac{I}{\langle f \rangle +J(f)} \supset \frac{I^2 + \langle f \rangle +J(f)}{\langle f \rangle +J(f)} \supset \frac{I^3  + \langle f \rangle +J(f)}{\langle f \rangle +J(f)} \supset \cdots, 
    \end{equation*}
    and since $j(f)<\infty$ then $\frac{I}{\langle f \rangle +J(f)}$ has length $j_I(f)$. Therefore the sequence must terminate at the $j_I(f)-$th element, and so $\frac{I^{j_I(f)} + \langle f \rangle +J(f)}{\langle f \rangle +J(f)}=0$, i.e. $I^{j_I(f)} \subset \langle f \rangle +J(f)$.
    For the third item, If $0 \neq f \in I^{\ord_I(f)}$ then by Lemma~\ref{lem:app_ord_der} we can conclude that $\langle f \rangle +J(f) \subset I^{\ord_I(f)-1}$, which induces a surjective map of $R-$modules
    \begin{equation*}
        \frac{I}{\langle f \rangle +J(f) } \geq \frac{I}{I^{\ord_I(f)-1}}.
    \end{equation*}
    By applying length over $R[[\underline{x}]]$ to both sides of the equation we get that $j(f)$ is bigger than the $R[[\underline{x}]]-$length of  $\frac{I}{I^{\ord_I(f)-1}}$, which equals to $\binom{s+\ord_I(f)-2}{s+1}$ (as it is the number of monomials in $s$ variables of positive degree at most $\ord_I(f)-2$, where $s$ corresponds to a sequence of length $s$ that generate $I$).
\end{proof}

\begin{remark}
\begin{enumerate}
    \item \textup{Item 2 of Proposition~\ref{prop:jacobian_prop} gives us a quantitative bound on the determinacy of $f$ with respect to $I$. It is analogue of the role of the Milnor and Tjurina numbers in Corollary 2.4 of~\cite{boubakri2012invariants} and of the role of the $c_I(f)$ invariant in Theorem 1.6 in~\cite{de1988some}, Theorem 6.5. in~\cite{pellikaan1988finite}, and Proposition 1.6 in~\cite{siersma1983isolated}). }
    \item  \textup{Since we are looking at any ideal $I \subset R[[\underline{x}]]$, the enumerative behavior of $j_I$ becomes very complicated the more "complicated" the ideal is $I$. For example, as an analogue of Morse's Lemma (see Theorem 2.46 in~\cite{greuel2007introduction}) and of Remark 2.16 in~\cite{pellikaan1990deformations}) one expect that (in the case where the characteristic of $\frac{R}{\mathfrak{m}}$ is not $2$) if $I=\langle a_1, \dots, a_n \rangle$ and $f=a_1^2 +\cdots +a_n^2$ then we have that $J(f)=I$, yet, since $\partial_i(f)=\sum_{j=1}^n 2a_j \partial_i(a_j)$, then by Lemma 4.2 in~\cite{greuel2019finite}, this is true if and only if the Jacobian matrix $D\psi \in \Mat_r(R[[\underline{x}]])$ of the $R-$homorphism $\psi$ of $R[[\underline{x}]]$ defined by $\psi(x_i)=a_i$ for every $i$ is invertible, which by Lemma~\ref{lem:W[[x]]_good_prop}, this is true if and only if $\psi$ is an automorphism, which only happens when $I=\mathfrak{M}$.}
\end{enumerate}
\end{remark}

We can use the determinacy theorem with respect to an ideal to conclude a version of the Mather-Yau theorem over $R[[\underline{x}]]$ with respect to an ideal for finitely determined elements. For a zero characteristic version of the Mather-Yau theorem, see Theorem 2.26 in~\cite{greuel2007introduction} or~\cite{mather1982classification}, and for a positive characteristic version, see~\cite{greuel2013mather}. 

\begin{proposition}[Mather-Yau Theorem over $I$]\label{prop:mather-yau}
    Let $I \subset R[[\underline{x}]]$ be an ideal and let $f, g \in I^2 \subset R[[\underline{x}]]$ be such that $f$ has finite determinacy with respect to $I$. Then $f \equivd g$ if and only if for every $k\gg 0$ we have an isomorphism of $R-$algebras
    \begin{equation*}
        \frac{R[[\underline{x}]]}{\langle f, I^k\cdot  J(f) \rangle} \cong \frac{R[[\underline{x}]]}{\langle g, I^k\cdot  J(g)\rangle}.
    \end{equation*}
\end{proposition}

\begin{proof}
    Assume that $f \equivd g$. Then there exists a unit $u \in R[[\underline{x}]]$ and an automorphism $\varphi \in \mathcal{D}_I$ of $R[[\underline{x}]]$ such that $g=u \varphi(f)$. Then by Lemma~\ref{lem:aut_der} we have that 
    \begin{equation*}
        \langle g, I^k J(g) \rangle = \langle u \varphi(f), I^k J\left(u \varphi(f)\right) \rangle = \langle \varphi(f), I^k J\left(\varphi(f)\right) \rangle = \varphi\left(\langle f, I^k J(f) \rangle\right)
    \end{equation*}
    for every $k$, which gives us the desired isomorphism of $R-$rings. \\

    Assume that for every big enough $k\gg 0$ we have an isomorphism  of $R-$rings $\frac{R[[\underline{x}]]}{\langle f, I^k\cdot  J(f) \rangle} \cong \frac{R[[\underline{x}]]}{\langle g, I^k\cdot  J(g)\rangle}$. Therefore, by the Lemma~\ref{lem:lifting}, this isomorphism lifts to an automorphism $\varphi \colon R[[\underline{x}]] \to R[[\underline{x}]]$ such that $\langle g, I^k J(g) \rangle= \varphi\left(\langle f, I^k J(f) \rangle\right)$. By Lemma~\ref{lem:aut_der} we have that 
    \begin{equation*}
        \langle \varphi(f), I^k J\left(\varphi(f)\right) \rangle = \varphi\left(\langle f, I^k J(f) \rangle\right),
    \end{equation*}
    and so we can assume (up to $\equivd$) that $\langle g, I^k J(g) \rangle= \langle f, I^k J(f) \rangle$. Therefore, there exists some $H_1 \in R[[\underline{x}]]$ and $H_2 \in I^k J(g)$ such that $f  = H_1g + H_2$. By Lemma~\ref{lem:app_ord_der} we have that $J(f) \subset I^{\ord_I\left(f\right)-1}$ and so $I^k J(f) \subset I^{k+\ord_I\left(f\right)-1} \subset I^{\ord_I\left(f\right)}$. Therefore $g \in I^{\ord_I\left(f\right)}$ and $H_2 \in I^{\ord_I\left(f\right) + k-1}$. \\

    We show that $H_1$ must be a unit. Assuming otherwise, then since we have that $\langle g, I^k J(g) \rangle= \langle f, I^k J(f) \rangle$, we can conclude that $g=H_3 f + H_4$ for some $H_3 \in R[[\underline{x}]]$ and $H_4 \in I^k J(f)$. Therefore, since $f=H_1g+H_2$ we have that $\left(1- H_1 H_3\right) f = \left(H_2 H_3\right)+H_4 \in I^{k+\ord_I\left(f\right)-1}$. Since $H_1$ is not a unit then $1- H_1 H_3$ must be a unit and we can conclude that $f \in I^{k+\ord_I\left(f\right)-1}$. This is impossible since $k+\ord_I\left(f\right)-1 > \ord_I\left(f\right)$. Therefore we must have that $H_1$ is a unit, and so $f - H_1 g \in I^{\ord_I\left(f\right)+k-1}$. Since $f$ has finite determinacy with respect to $I$ and since $k\gg 0$ (which we can assume is bigger than the determinacy order of $f$), then we have that $H_1 g \equivd f$. Yet, $H_1 g \equivd g$ and so the result follows. 
\end{proof}

\begin{remark}
    \textup{Note that the fact that $\langle f \rangle + J(f) = \langle g \rangle + J(g)$ does not give us that $f \equivk g$, as in the version of Mather-Yau theorem over $\mathbb{C}$ (see, for example, Theorem 2.26 in~\cite{greuel2007introduction}). If the characteristic of $R$ is $p>0$, set $f=x_1^{p+1}+x_2^{p+1}$ and $g=x_1^p + x_1^{p+1}+x_2^{p+1}$. Then we have that $\langle f \rangle + J(f) = \langle g \rangle + J(g)$ but $f \not\equivk g$ as they have different orders.}
\end{remark}

We end this section with a discussion on unfoldings over $R$. Unfoldings play a crucial role in singularity theory over fields, and for a deeper discussion see Chapter II of~\cite{greuel2007introduction} or Section 2 of~\cite{greuel2016right}.

\begin{definition}
    Let $f \in R[[\underline{x}]]$. An \textbf{unfolding} of $f$ is an element $f_t = f\left(\underline{x},t\right) \in R[[\underline{x}]][t]$ of the form $f_t = f +tg$ for some $g \in R[[\underline{x}]]$. Given some $t_0 \in \mathfrak{m}$ and some unfolding $f_t=\sum_{j=0}^\infty\sum_{\alpha} a_{{\alpha},j}\underline{x}^{\alpha}t^j \in R[[\underline{x}]][t]$ of $f$, we set $f_{t_0} = \sum_{j=0}^\infty\sum_{\alpha} a_{{\alpha},j}\underline{x}^{\alpha} t_0^j \in R[[\underline{x}]]$.
\end{definition}

\begin{proposition}\label{prop:semi}
    Let $f \in R[[\underline{x}]]$. Then:
    \begin{enumerate}
        \item For every unfolding $f_t$ of $f$ there exists some $N$ such that for every $t_0 \in \mathfrak{m}^N$ we have that $\ord\left(f\right) = \ord\left(f_{t_0}\right)$.
        \item If $f$ has finite determinacy with respect to $\mathfrak{M}$ then for every unfolding $f_t$ of $f$ there exists some $N$ such that for every $t_0 \in \mathfrak{m}^N$ we have that $f \equivk f_{t_0}$. 
    \end{enumerate}

\end{proposition}

\begin{proof}
For the first item, we can write $f_t = f + tg \in R[[\underline{x}]][t]$ for some $g \in R[[\underline{x}]][t]$. If we set $N=\ord\left(f\right)+1$ then for every $t \in \langle t_0 \rangle^N$ we must have that $\ord\left(f\right)=\ord\left(f_{t_0}\right)$, as $tg \in \mathfrak{M}^N$. For the second item, if $f$ is $k-$determined with respect to $\mathfrak{M}$ then for every $N>k$ we have that if $t_0 \in \mathfrak{M}^N$ then $f -f_{t_0} \in \mathfrak{M}^{k+1}$, and therefore $f \equivk f_{t_0}$.  
\end{proof}

\begin{remark}\label{rem:unfolding}
\begin{enumerate}
    \item \textup{The results in Proposition~\ref{prop:semi} are much stronger than their counterparts over a field. That is, if $f_t$ is a unfolding of $f \in k[[\underline{x}]]$ (for some complete topologized field $k$), then in a neighborhood $U$ of $t=0$ we can only guaranty the inequalities $\mu\left(f_0\right) \geq \mu\left(f_{t_0}\right)$, $\tau\left(f_0\right) \geq \tau\left(f_{t_0}\right)$, and $\ord\left(f_0\right) \leq \ord\left(f_{t_0}\right)$, for every $t_0 \in U$ (where $\tau$ and $\mu$ are the Tjurina and Milnor number over $k$, respectively). This strengthened results come from the fact that over $R$, the variable $t$ can affect the order of $f_t$ but over a field they do not. For more information on the field case, see Appendix A in~\cite{greuel2016right}. } 
    \item \textup{Inspired by Section 3 of~\cite{greuel2017singularities} and Section I 2.4 of~\cite{greuel2007introduction}, one can define that $f \in R[[\underline{x}]]$ has \textbf{finite deformation type} if there exists a finite collection of elements $g_1, ..., g_l \in R[[x]]$ such that for every unfolding $f_t \in R[[x]][t]$ of $f$  there exists some $k$ such that for every $t_0 \in \mathfrak{m}^k$ there exists some $i$ such that $f_{t_0} \equivk g_i\left(x\right)$. In the field case we know that finite deformation type singularities are exactly $ADE$ (for more information on $ADE$ and deformations, see~\cite{greuel1990simple, greuel2017singularities}). But over a general local ring, Proposition~\ref{prop:semi} gives us in fact that every element of $R[[\underline{x}]]$ of finite determinacy with respect to $\mathfrak{M}$ has finite deformation type, which by Corollary~\ref{cor:sqrtI}, is true for every $f \in R[[\underline{x}]]$ such that $\sqrt{\langle f \rangle + J(f)} = \mathfrak{M}$.}
\end{enumerate}
 \end{remark}

\section{Power Series over a DVR}\label{sec:V_power_series}

In this section we focus on on the ring of power series $(\Vx, \maxV, \kappa)$ over a DVR $V$ with a chosen uniformizer $\pi$ (where $\maxV=\langle \pi, x_1, \dots, x_n\rangle$). The goal of this section is to develop an analogue of the Tjurina and the Milnor number and of $\equivk$ over $\Vx$ and to show that they have similar properties by "viewing $\pi$ as a variable", inspired by the discussion presented in Section~\ref{Introduction}. We formalize this idea using the following Notational setup:

\begin{notation}\label{def:subsoil}
Rcall that given some $a \in V$, there exists a unique integer $n \geq 0$ and a unique unit $u \in V$ such that $a=u\pi^n$. Therefore, given some $f \in V[[\underline{x}]]$, if we can write $f=\sum_{\alpha}a_{\alpha} \underline{x}^{\alpha}$, then $a_{\alpha}$ can be written uniquely as $u_{\alpha} \pi^{n_{\alpha}}$, and thus $f=\sum_{\alpha} u_{\alpha} \pi^{n_{\alpha}} \underline{x}^{\alpha}$. In this case, we denote 
 \begin{equation*}
     \tilde{f}(\underline{x},y)=\sum_{\alpha} u_{\alpha} y^{n^{\alpha}} \underline{x}^{\alpha} \in V[[\underline{x},y]],
 \end{equation*}
  where $V[[\underline{x},y]]=V[[x_1, \dots, x_n, y]]$ is the ring of formal power series in variables $x_1, \dots, x_n,y$ whose maximal ideal is $\naxV=\langle p, x_1, \dots, x_n, y \rangle$.
\end{notation}

\begin{remark}\label{rem:y-p}
\begin{enumerate}
    \item \textup{Note that the map $\Vx \to \Vxy$ defined by $f \mapsto \tilde{f}$ respects neither addition nor multiplication in general. Therefore, it can be thought of as a set-theoretical section of the projection $\text{pr} \colon \Vxy \to \Vx$ defined by $\text{pr}\left(y\right)=\pi$ and $\text{pr}\left(x_i\right)=x_i$ for every $i$. }
    \item \textup{The choice of uniformizer $\pi$ of $V$ does not effect the construction of $\tilde{f}$ up to $\equivr$. Given two uniformizers $\pi_1$ and $\pi_2$ of $V$, then there exists some unit $v \in V$ such that $\pi_2 = v \pi_1$. Therefore, given some $f=\sum_{\alpha} u_{\alpha} \pi^{n_{\alpha}} \underline{x}^{\alpha}$ as in Notation~\ref{def:subsoil}, then computing $\tilde{f}$ with respect to $\pi_1$ gives us $\tilde{f}_1=\sum_{\alpha} u_{\alpha} y^{n_{\alpha}} \underline{x}^{\alpha}$ and computing $\tilde{f}$ with respect to $\pi_2$ gives us $\tilde{f}_2=\sum_{\alpha} u_{\alpha} (vy)^{n_{\alpha}} \underline{x}^{\alpha}$.  Thus $\tilde{f}_1 \equivr \tilde{f}_2$ by the automorphism $y \mapsto vy$.  }
\end{enumerate}

\end{remark}

Inspired by Notation~\ref{def:subsoil}, we define a version of the Tjurina and Milnor number of $f$ over $\Vx$ based on $\tilde{f}$, keeping in mind the idea that we want to view the uniformizer $\pi$ as a variable. In order to do so, we would like to look at the colength of $J(\tilde{f})+\langle \tilde{f}\rangle$. But, this ideal is of at least coheight 1 in general, and so must have infinite length. Yet, the following proposition will tell us that if we intersect it with a "generic" hypersurface, it will have a constant finite colength. \\

Recall that given an ideal $I$, we say that an ideal $J \subset I$ is a minimal reduction of $I$ if there exists some $N$ such that $JI^N = I^{N+1}$ and if there is no ideal $J_1 \subsetneq J$ with this property (For more information, see Chapter 8 in~\cite{huneke2006integral}).  

\begin{proposition}\label{prop:generic_intersection}
    Let $I \subset \Vxy$ be an ideal such that $\frac{\Vxy}{I}$ is a one-dimensional ring. Then:
    \begin{enumerate}
        \item There exists a Zariski-open subset $U \subset \frac{\naxV}{\naxV^2 + I}$ such that for every $a \in \Vxy$, if $a + \left(\naxV^2 + I\right) \in U$ then $a \mod    I$ generates a minimal reduction of the maximal ideal of $\frac{\Vxy}{I}$. 
        \item If $\langle a \rangle \mod I$ is a minimal reduction of the maximal ideal of $\frac{\Vxy}{I}$ then we have that  $\textup{length}_{\Vxy}\left(\frac{\Vxy}{I+\langle a \rangle}\right)$ equals to the multiplicity of  $\frac{\Vxy}{I}$ (and therefore is finite and does not depend on the choice of $a$). 
        \item For every $a \in \Vxy$ we have that the multiplicity of $\frac{\Vxy}{I}$ is less than or equal to  $\textup{length}_{\Vxy}\left(\frac{\Vxy}{I+\langle a \rangle}\right)$.
    \end{enumerate}
\end{proposition}
\begin{proof}
    The first part is a special case of Theorem 8.6.6 in~\cite{huneke2006integral} (which in turn cites~\cite{northcott1954reductions, trung2003constructive}). The second part is a special case of Proposition 11.2.2 in~\cite{huneke2006integral}. The third part is a special case of semicontinuity of length (for more information, see~\cite{504741}).
\end{proof}

\begin{definition}\label{def:tjurinaV}
    Let $f \in \maxV \subset \Vx$. 
    \begin{enumerate}
        \item We define the \textbf{Milnor number of $f$} to be
        \begin{equation*}
            \mu_V\left(f\right)=\textup{length}_{\Vxy}\left(\frac{\Vxy}{J(\tilde{f}) + \langle a \rangle}\right)
        \end{equation*}
        for some $a \in \Vxy$ such that $ \langle a\rangle \mod     J(\tilde{f})$ is a minimal reduction of the maximal ideal of $\frac{\Vxy}{J(\tilde{f})}$, assuming that it is a one-dimensional ring. Otherwise, we set $\mu_V\left(f\right)=\infty$. 
        \item We define the \textbf{Tjurina number of $f$} to be
        \begin{equation*}
            \tau_V\left(f\right)=\textup{length}_{\Vxy}\left(\frac{\Vxy}{J(\tilde{f}) + \langle \tilde{f},  a \rangle}\right)
        \end{equation*} for some $a \in \Vxy$ such that $\langle a \rangle  \mod    (J(\tilde{f}) + \langle \tilde{f} \rangle)$ is a minimal reduction of the maximal ideal of $\frac{\Vxy}{J(\tilde{f}) + \langle \tilde{f} \rangle}$, assuming that it is a one-dimensional ring. Otherwise, we set $\tau_V\left(f\right)=\infty$. 
    \end{enumerate}
\end{definition}

\begin{remark}\label{rem:tjurina_def}
\begin{enumerate}
    \item \textup{Note that if $\mu_V\left(f\right) < \infty$ then $\tau_V\left(f\right) < \infty$ since $\tau_V\left(f\right) \leq \mu_V\left(f\right)$ (for a generic choice of $a$). But as in the positive characteristic case (see, for example, Remark 1.2.4. in~\cite{boubakri2009hypersurface}), the other direction is note true. Take, for example, $f\left(x\right)= x^p + p^{p-1}$ in the case where $p$ is the uniformizer. Yet, as in Proposition 2.1 of~\cite{hefez2019hypersurface}, we have that if $\tau_V\left(f\right) < \infty$ then $\mu_V\left(f\right) < \infty$ if and only if $\tilde{f} \in \sqrt{J(\tilde{f}) +\langle a \rangle}$ for a generic $a$ as in Definition~\ref{def:tjurinaV}. }%}
    \item \textup{We do not set the multiplicity of $\frac{\Vx}{\langle \tilde{f} \rangle + J(\tilde{f})}$ (resp. of $\frac{\Vx}{J(\tilde{f})}$) as the definition of $\tau_V\left(f\right)$ (resp. $\mu_V(f)$) since the multiplicity of $\frac{\Vxy}{I}$ will always be finite (for every ideal $I$), and as we see in Proposition~\ref{lem:tau(f,q)}, the finiteness of $\tau_V\left(f\right)$ as defined above tells us that $f$ has a special property that resembles having an isolated singularity. Note that over a field of positive characteristic $k$, one can define an alternative version of the Milnor number of $f \in k[[\underline{x}]]$ as  the multiplicity of $\frac{k[[\underline{x}]]}{J(f) + \langle f \rangle}$, and it behaves much more like its characteristic zero analogue. For more information, see Section 3 of~\cite{hefez2019hypersurface}.}
    \item \textup{Note that it is clear that the length of $\frac{\Vxy}{J(\tilde{f})}$ is infinite since $\Vxy$ is $(n+1)-$dimensional and $J(\tilde{f})$ is generated by $n$ elements, but it is not quite clear for $\frac{\Vxy}{J(\tilde{f}) + \langle \tilde{f} \rangle}$ (assuming that it is not zero). Yet, if it is non zero, by Proposition~\ref{prop:milnor_zero} we must have that $f\in \maxV^2$ and so $\tilde{f} \in \langle x_1, \dots, x_n, y \rangle^2$. Thus, by Lemma~\ref{lem:app_ord_der}, we have that $J(\tilde{f}) +\langle \tilde{f} \rangle \subset \langle x_1, \dots, x_n, y \rangle$.}
\end{enumerate}
\end{remark}

The following lemma gives us that $\tau_V$ detects a property of $f$ that resembles "being an isolated singularity" (but not fully), as a mixed characteristic analogue of Lemma 2.3 in~\cite{greuel2007introduction}. We discuss how to detect isolated singularities using an analogue of $\tau_V$ in depth in Section~\ref{sec:Jeffries_Hochter}.

\begin{proposition}\label{lem:tau(f,q)}
    Let $f \in \Vx$ and let $\tilde{\mathfrak{q}}$ be a prime ideal in $\Vxy$:
    \begin{enumerate}
        \item If $\left( \frac{\Vxy}{\langle \tilde{f} \rangle + J(\tilde{f})}\right) _{\tilde{\mathfrak{q}}}=0$ then $\left(\frac{\Vxy}{\langle \tilde{f} \rangle}\right)_{\tilde{\mathfrak{q}}}$ is regular.
    
    \item If $\tau_V\left(f\right) < \infty$ then for every prime ideal $\tilde{\mathfrak{q}}$ in $\Vxy$ such that $\frac{\Vxy}{\tilde{\mathfrak{q}}}$ is two dimensional we have that  $\left(\frac{\Vxy}{\langle \tilde{f} \rangle}\right)_{\tilde{\mathfrak{q}}}$ is regular. 
    \end{enumerate}
    
\end{proposition}

\begin{proof}
    For the first part, note that this quotient being zero is equivalent to having $\langle \tilde{f} \rangle + J(\tilde{f}) \not \subset \tilde{\mathfrak{q}}$. If $f \notin \tilde{\mathfrak{q}}$ then we are done, since then $\left(\frac{\Vxy}{\langle \tilde{f} \rangle}\right)_{\tilde{\mathfrak{q}}}=0$. Otherwise, if $\tilde{f} \in \tilde{\mathfrak{q}}^2$ then by Lemma~\ref{lem:app_ord_der} we have that $J(\tilde{f}) \subset \tilde{\mathfrak{q}}$.  Therefore we must have that $\tilde{f} \in \tilde{\mathfrak{q}} \setminus \tilde{\mathfrak{q}}^2$, and so $\left(\frac{\Vxy}{\langle \tilde{f} \rangle}\right)_{\tilde{\mathfrak{q}}}$ is regular.\\

    For the second part, if $\tau_V\left(f\right) < \infty $ then by Remark~\ref{rem:tjurina_def} we have that $\frac{\Vxy}{\langle \tilde{f} \rangle + J(\tilde{f})}$ must be one-dimensional. Thus, since  $\frac{\Vxy}{\tilde{\mathfrak{q}}}$ is two dimensional, we must have that $\langle \tilde{f} \rangle + J(\tilde{f}) \not \subset \tilde{\mathfrak{q}}$, and so by the previous part we must have that
    $\left(\frac{\Vxy}{\langle f \rangle}\right)_{\tilde{\mathfrak{q}}}$ is regular.
\end{proof}

\begin{remark}
\begin{enumerate}
    \item \textup{Note that the reverse direction of Proposition~\ref{lem:tau(f,q)} is false. For example, if $p$ is uniformizer of $V$, let $f=p^2+x_1^p+x_2^p+x_3^p \in V[[x_1, x_2, x_3]]$, and choose $\tilde{\mathfrak{q}}=\langle y , x_1^p+x_2^p+x_3^p \rangle$. Then $J(\tilde{f}) \subset \tilde{\mathfrak{q}}$ which tells us that $\left(\frac{\Vxy}{\langle \tilde{f} \rangle + J(\tilde{f})}\right)_{\tilde{\mathfrak{q}}} \neq 0$ but since $\tilde{f} \in \tilde{\mathfrak{q}} \setminus \tilde{\mathfrak{q}}^2$ we have that $\left(\frac{V[[x_1, x_2,x_3,y]]}{\langle \tilde{f} \rangle}\right)_{\tilde{\mathfrak{q}}}$ is regular. } 
    \item \textup{Note that from Lemma~\ref{lem:tau(f,q)} we can conclude that if $\tau_V(f) < \infty$ then for every $y-\pi 
    \notin \tilde{\mathfrak{q}} \subset \Vxy$ we have that $\left(\frac{\Vx}{\langle f \rangle}\right)_{\mathfrak{q}}$ is regular where $\mathfrak{q} = \tilde{\mathfrak{q}} \mod (y-\pi)$.}
\end{enumerate}
\end{remark}

The following Corollary is an analogue of item 3 of Proposition~\ref{prop:jacobian_prop}. 

\begin{proposition}\label{cor:tjurina_bound}
    Let $f \in \Vx$ be of order $s \geq 2$. Then we have that
    \begin{equation*}
        \mu_V\left(f\right) \geq \tau_V\left(f\right) \geq \binom{n+s-2}{n+1}.
    \end{equation*}
\end{proposition}

\begin{proof}
By Lemma~\ref{lem:app_ord_der} we have that $\langle \tilde{f} \rangle+J(\tilde{f}) \subset \langle x_1, \dots, x_n, y \rangle^{s-1}$, and so we have a surjective map 
\begin{equation*}
   \frac{\Vxy}{\langle \tilde{f} \rangle+J(\tilde{f})} \to   \frac{\Vxy}{\langle x_1, \dots, x_n, y \rangle^{s-1}}.  
\end{equation*}
 Therefore, by Proposition~\ref{prop:generic_intersection} we can conclude that $\tau_V\left(f\right)$ is bigger than the multiplicity of the module $\frac{\Vxy}{\langle x_1, \dots, x_n, y \rangle^{s-1}}$ over the maximal ideal $\naxV$, which equals to $\binom{n+s-2}{n+1}$ (as it is exactly the number of monomials in $x_1, \dots, x_n, y$ of degree at most $s-2$).    
\end{proof}

Inspired by Section~\ref{appendix_Det} and by Notation~\ref{def:subsoil}, we now turn to defining an analogue of $\equivk$ that allows us to "view $\pi$ as a variable" by moving from $f$ to $\tilde{f}$.

\begin{definition}\label{def:equiv}
Given some $f,g \in \Vx$. We say that $f$ and $g$ are equivalent, denoted $f \sim g$, if $\tilde{f} \equivk \tilde{g}$ as elements of $\Vxy$. 
\end{definition}

The following proposition allows us to relate between equivalence and "change of variables" (especially in the case where we set $g=\tilde{g_1}$ for some $g_1 \in \Vx$), as viewed, for example, in the classification process of Theorem B in~\cite{carvajal2019covers} and in the notion of "can be written as", as presented in Definition 3.10 of~\cite{svoray2025ade}. 

\begin{proposition}\label{prop:equivalence_easy}
    Given $f \in \Vx$ and $g \in \Vxy$, if $\tilde{f} \equivk g$ then there exists a sequence $a_1, \dots, a_{n+1}$ that generates $\maxV$ and there exists some $H \in V[[z_1, \dots, z_{n+1}]]$ (the ring of formal power series in  some dummy variables $z_1, \dots, z_{n+1}$) whose coefficients are units such that
    \begin{equation*}
         f =  H\left(\pi, x_1, \dots, x_n\right)  \text{ and } \langle \text{pr}(g) \rangle= \langle H\left(a_1, \dots, a_{n+1}\right) \rangle.
    \end{equation*}
    In addition, if $f\left(\underline{x}\right)=\sum_{\alpha} u_{ \alpha} \pi^{n_{\alpha}} \underline{x}^{\alpha}$ as in Notation~\ref{def:subsoil} then for $\underline{z}=(z_1, \dots, z_n)$ we can choose $H\left(z_1, \dots, z_{n+1}\right)=\sum_{\alpha} u_{ \alpha} z_{n+1}^{n_{\alpha}} \underline{z}^{\alpha}$.
\end{proposition}

\begin{proof}
     Then there exists some automorphism $\varphi$ of $\Vx$ and some unit $u \in \Vxy$ such that $ug=\varphi(\tilde{f})$. If we write $f$ as $f\left(\underline{x}\right)=\sum_{\alpha} u_{ \alpha} \pi^{n_{\alpha}} \underline{x}^{\alpha}$ as in Notation~\ref{def:subsoil}, then we have that  $\varphi(\tilde{f}) =\sum_{\alpha} u_{\alpha} \varphi\left(y\right)^{n_{\alpha}}\varphi\left(\underline{x}\right)^{\alpha}$. Since $\text{pr}(g) = \text{pr}\left(u^{-1}\varphi(\tilde{f})\right)=\text{pr}(u^{-1})\text{pr}\left(\varphi(\tilde{f})\right)$ and since $\text{pr}\left(u^{-1}\right)$ has to be a unit in $\Vx$, then we can conclude that 
    \begin{equation*}
        \langle \text{pr}(g) \rangle = \langle \sum_{\alpha} u_{\alpha} \text{pr}\left(\varphi\left(y\right)\right)^{n_\alpha} \text{pr}\left(\varphi\left(\underline{x}\right)\right)^{\alpha} \rangle.
    \end{equation*}
     Since $\varphi$ is an isomorphism then we must have that that the ideal defined by $\langle \text{pr}\left(\varphi\left(x_1\right) \right), \dots, \text{pr}\left(\varphi\left(x_n\right)\right), \text{pr}\left(\varphi\left(y\right)\right) \rangle$ is exactly $\maxV$.
    Therefore we can set $H\left(z_1, \dots, z_{n+1}\right)=\sum_{\alpha} u_{ \alpha} z_{n+1}^{n_{\alpha}} \underline{z}^{\alpha}$ where $\underline{z}=\left(z_1, \dots, z_n\right)$ together with $a_i=\text{pr}(\varphi(x_i))$ for $i=1,\dots, n$ and $a_{n+1}=\text{pr}(\varphi(y))$, and the result follows. 
\end{proof}

%\begin{remark}\label{rem:H_setting}
%\textup{Note that by the proof of Proposition~\ref{prop:equivalence_easy} we can assume that if $f\left(\underline{x}\right)=\sum_{j, \alpha} u_{j, \alpha} p^j \underline{x}^{\alpha}$ where $u_{j, \alpha}$ are multiplicative lifts then  $H\left(z_1, \dots, z_{n+1}\right)=\sum_{j, \alpha} u_{j, \alpha} z_{n+1}^j \underline{z}^{\alpha}$. This observations is useful when we want to prove equivalences with specific forms, such as the ADE classification in Section~\ref{sec:ADE}.}
%\end{remark}

The following proposition is a mixed characteristic version of Lemma 2.10 in~\cite{greuel2007introduction} and of Lemma 1.2.7 in~\cite{boubakri2009hypersurface}.

\begin{proposition}\label{prop:order_equivalence}
    If $f \sim g$ then $\ord\left(f\right)=\ord\left(g\right)$ and $\tau_V\left(f\right) = \tau_V\left(g\right)$.

    %\begin{enumerate}
     %   \item  
        %\item If $f$ defines an isolated singularity and $f \sim g$ then so does $g$.
    %\end{enumerate}
    
\end{proposition}

\begin{proof}
         If $f \sim g$ then there exists a unit $u \in \Vxy$ and some automorphism $\varphi$ such that $\tilde{f}=u \varphi(\tilde{g})$. First, If $g \in \maxV^2 \setminus \maxV^{r+1}$ then $\tilde{g} \in \naxV^2 \setminus \naxV^{r+1}$. Therefore $\tilde{f}=u \varphi(\tilde{g}) \in \naxV^2 \setminus \naxV^{r+1}$ and so $f \in \maxV^2 \setminus \maxV^{r+1}$. Second, by Lemma~\ref{lem:aut_der} we have that $\langle \tilde{f}\rangle +J(\tilde{f})=\langle \tilde{g}\rangle+J\left(\tilde{g}\right)$ and the result follows from Proposition~\ref{prop:generic_intersection} (for a generic choice of $a \in \Vxy$ as in Definition~\ref{def:tjurinaV}). 
\end{proof}

\begin{remark}\label{rem:equiv_non_iso}
\textup{As in the positive characteristic case, if $f \sim g$ then we can not conclude that $\mu_V\left(f\right)=\mu_V\left(g\right)$. For example, set $f\left(x\right)=x^p + p^{p-1} \in V[[x]]$, assuming $p$ is a uniformizer of $V$. Then $\mu_V\left(f\right) = \infty$ since $\frac{\Vxy}{J(\tilde{f})}$ defines a ring of dimension 2, but $\mu_V\left(\left(1+x\right)f\right)$ is finite since we have that $\sqrt{J\left(\left(1+x\right)\left(x^p + y^{p-1}\right)\right)}=\langle x,y \rangle$.  Yet, as we see in Remarks~\ref{rem:smooth} and~\ref{rem:Morse_milnor}, there are some very specific cases where this is true. }
\end{remark}

The following lemma tells us that $\equivk$ over $V[[\underline{x}]]$ induces the equivalence $\sim$. 

\begin{lemma}\label{lemma:equiv_iso}
    Let $f,g \in \Vx$ and assume $\frac{\Vx}{ \langle f \rangle} \cong \frac{\Vx}{ \langle g \rangle}$ as $V-$algebras. Then $f \sim g$.
\end{lemma}

\begin{proof}
    If $\frac{\Vx}{ \langle f \rangle} \cong \frac{\Vx}{ \langle g \rangle}$ then by Lemma~\ref{lem:W[[x]]_good_prop} there exists some automorphism $\varphi$ of $\Vx$ such that $\langle \varphi\left(f\right) \rangle = \langle g \rangle$. Therefore, we can extend $\varphi$ to $\Vxy$ by setting $\varphi\left(y\right)=y$, which would give us that $\langle \varphi(\tilde{f}) \rangle = \langle \tilde{g} \rangle$, and the result follows. 
\end{proof}

The following proposition shows that regular hypersurfaces are exactly of order $1$ and have a specific form up to equivalence (as a mixed characteristic version of Lemma 2.44 in~\cite{greuel2007introduction}). 

%Note the similarity of the proof of Proposition~\ref{prop:milnor_zero} with the proof of Proposition~\ref{lem:tau(f,q)}. 

\begin{proposition}\label{prop:milnor_zero}
    Let $f \in \maxV \subset \Vx$. Then the following are equivalent:
    \begin{enumerate}
        \item $\frac{\Vx}{\langle f \rangle}$ is a regular local ring, 
        \item $f$ is of order $1$, 
        \item $f \sim x_1$,
        \item $\tau_V\left(f\right)=0$.
    \end{enumerate}
\end{proposition}

\begin{proof}
    For $1 \to 2$, if $S=\frac{\Vx}{\langle f \rangle}$ is regular (with maximal ideal $\mathfrak{m}_S$) and $f \in \maxV^2$ then we would have that
    \begin{equation*}
        \dim_\kappa\left(\frac{\maxV}{\maxV^2}\right) =  \dim_\kappa\left(\frac{\mathfrak{m}_S}{\mathfrak{m}_S^2}\right) \neq \dim \Vx,
    \end{equation*}
    \noindent which is a contradiction. For $2 \to 1$, if $f$ is of order $1$ then $f \in \maxV \setminus\maxV^2$, and so $f$ is a non zero divisor. Therefore, since $\Vx$ is regular, then so is $\frac{\Vx}{\langle f \rangle}$.  For $2 \to 3$, if $f$ is of order $1$ then $\tilde{f}$ must be of order $1$ (in $\Vxy$), and so we can write $\tilde{f}  = uy + u_1x_1 + \cdots +u_n x_n +g$ for some $g \in \langle x_1, \dots, x_n, y \rangle^2$ and some $u, u_1, \dots, u_n \in V$ that are either zero or units. So, by looking at the $V-$automorphism of $\Vxy$ defined by $x_1 \mapsto uy + u_1x_1 + \cdots u_n x_n$, we can conclude that $\tilde{f} \equivk x_1 + g_1$ for some $g_1 \in \langle x_1, \dots, x_n, y \rangle^2$. Therefore by the implicit and the inverse function theorem (Lemma~\ref{lem:W[[x]]_good_prop}) we get that $f \sim x_1$. For $3 \to 4$, if $f \sim x_1$ then by Proposition~\ref{prop:order_equivalence} we can conclude that $\tau_V\left(f\right)=\tau_V\left(x_1\right)=0$. For $4 \to 2$,  assume that $\tau_V \left(f\right)=0$. Then $1 \in \langle \tilde{f},a \rangle + J(\tilde{f})$ for some $a$ as in Proposition~\ref{prop:generic_intersection}. Since $a\in \naxV$, we must have that $1 \in \langle \tilde{f} \rangle + J(\tilde{f})$ and so there exists some $b_1, \dots, a_n,b, c \in \Vxy$ such that 
    \begin{equation*}
        1 = b\partial_y(\tilde{f}) + c\tilde{f} + \sum_i b_i \partial_i(\tilde{f}).
    \end{equation*}
     \noindent But since $f \in \maxV$, then either $\partial_y(\tilde{f}) \notin \naxV$ or there exists some $i$ such that $\partial_i(\tilde{f}) \notin \naxV$. Then by Lemma~\ref{lem:app_ord_der} we can conclude that that $\tilde{f}$ is of order $1$ and therefore so is $f$.
\end{proof}

\begin{remark}\label{rem:smooth}
%\begin{enumerate}
    %\item \textup{Note that by a proof similar to that of Proposition~\ref{lem:tau(f,q)} one can show that given some $p \notin \mathfrak{q} \subset \Vx$ prime ideal, then $\left(\frac{\Vx}{\langle f \rangle}\right)_\mathfrak{q}$ is regular if and only if $\langle f \rangle + J\left(f\right) \not\subset \mathfrak{q}$. This can be though of as a partial version of mixed characteristic Jacobian criterion, and was first presented in~\cite{nagata1958general}. Jeffries and Hochter in~\cite{hochster2021jacobian}, proved a complete version of the Jacobian criterion in the mixed characteristic, in addition to constructing an  analogue to derivations. We discuss some of their work in Section~\ref{sec:Jeffries_Hochter}.}
    %\item
    \textup{Note that given $f \in \maxV$, since $\mu_V\left(f\right) \geq \tau_V\left(f\right)$ then $\mu_V\left(f\right)=0$ gives us that $\tau_V\left(f\right)=0$ as well. In fact, the reverse is also true. It $\tau_V\left(f\right)=0$ then by the proof of Proposition~\ref{prop:milnor_zero} we must have that $J(\tilde{f})=\Vxy$, and so $\mu_V(f)=0$.}
%\end{enumerate}
\end{remark}

We turn to a discussion on determinacy over $V$, based upon the results of Section~\ref{appendix_Det}. 

\begin{definition}\label{def:det_V}
      We say that $f \in \Vx$ is \textbf{$k-$determined} if for every $g \in \Vx$, if $\tilde{f} - \tilde{g} \in \langle x_1, \dots, x_n, y \rangle^{k+1}$ then $f \sim g$. We say that $f$ has \textbf{finite determinacy} if it is $k-$determined for some $k$.  
\end{definition}

Note that a similar definition can be defined for elements in $\Vxy$. 

\begin{remark}\label{rem:determinacy}
%\begin{enumerate}
 %   \item \textup{Note we can write $f\left(\underline{x}\right)=\sum_{j, \alpha} u_{j, \alpha} \pi^j \underline{x}^{\alpha}$ as in Notation~\ref{def:subsoil}, and so we view $jet_k\left(f\right)$ as the polynomial element defined by $\sum_{j+ |\alpha| \leq k} u_{j, \alpha} \pi^j \underline{x}^{\alpha}$. Note that this gives us in fact that if $jet_k(f)=jet_k(g)$ then $\tilde{f} - \tilde{g} \in \langle x_1, \dots, x_n,y \rangle^{k+1}$. }
     \textup{Note that $\tilde{f}$ is $k-$determined (as in Section~\ref{appendix_Det}) then we can conclude that  ${f}$ is $k-$determined. Thus, by Theorem~\ref{thm_App} we have the following version of the finite determinacy for $\Vx$: If $f \in \maxV^2 \subset \Vx$ with
    \begin{equation*}
        \langle x_1, \dots, x_n,y \rangle^{k+2} \subset \langle x_1, \dots, x_n,y \rangle \cdot \langle \tilde{f} \rangle + \langle x_1, \dots, x_n,y \rangle^2 \cdot J(\tilde{f}) \subset \Vxy,
    \end{equation*}
      then $f$ is $\left(2k-\ord\left(f\right)+2\right)-$determined. }
%\end{enumerate}
\end{remark}

The following proposition is an analogue of Corollary 2.24 in~\cite{greuel2007introduction} and of Lemma 2.6 in~\cite{greuel1990simple}.

\begin{proposition}\label{lem:reduc_a}
    Let $f \in \Vx$ with $\langle x_1, \dots, x_n,y \rangle \subset  \sqrt{\langle \tilde{f} \rangle + J(\tilde{f})}$ and denote by $\overline{f} = \tilde{f} \mod \pi$, which we view as an element of $\kappa[[\underline{x},y]]$. Then:
    \begin{enumerate}
        \item $\tau_V\left(f\right) < \infty$,
        \item $f$ is finitely-determined,
        \item $\overline{f}$ defines an isolated singularity over $\kappa[[\underline{x},y]]$.
    \end{enumerate} 
\end{proposition}

\begin{proof}
    Since  $\langle x_1, \dots, x_n,y \rangle \subset  \sqrt{\langle \tilde{f} \rangle + J(\tilde{f})}$ then we can find some $k$ for which we have that $\langle x_1, \dots, x_n,y \rangle^k \subset \langle \tilde{f} \rangle +  J(\tilde{f})$, hence $f$ has finite determinacy by Proposition~\ref{thm_App} and Remark~\ref{rem:determinacy}. In addition, $\frac{\Vxy}{\langle \tilde{f} \rangle + J(\tilde{f})}$ is a one-dimensional ring, and so by Remark~\ref{rem:tjurina_def} we must have that $\tau_V\left(f\right) < \infty$. Now, since $\langle x_1, \dots, x_n,y \rangle^k \subset \langle \tilde{f} \rangle +  J(\tilde{f})$ then going$\mod    \pi$ gives us that $\langle x_1, \dots, x_n, y \rangle^k \subset \langle \overline{f} \rangle +  J\left(\overline{f}\right)$ in $\kappa[[\underline{x},y]]$. Therefore we can conclude that $\overline{f}$ defines an isolated singularity (for more details, see Section 1.2 of~\cite{boubakri2009hypersurface}).
\end{proof}

\begin{remark}
\begin{enumerate}
    \item \textup{Note that not every $f \in \Vx$ with $\tau_V\left(f\right) < \infty$ has the property that $\langle x_1, \dots, x_n,y \rangle \subset  \sqrt{\langle \tilde{f} \rangle + J(\tilde{f})}$. For example, if $p$ is the uniformizer of $V$, then for $f\left(x\right)=x^p+p^p$ we have that the radical of  ${\langle \tilde{f} \rangle + J(\tilde{f})}$ equals to $\langle py,  x+y \rangle$, which does not contain $\langle x,y \rangle$. }
    \item \textup{Using the notations of Proposition~\ref{lem:reduc_a}, note that if $f \sim g$ then $\overline{f} \equivk \overline{g}$ over $\kappa[[\underline{x}]]$. However, the opposite is not true. For example, if $\kappa$ is of characteristic $p$ then $x^p+y^p$ and $y^p$ are contact equivalent over $\kappa[[x,y]]$ (since $x^p+y^p=(x+y)^p$ over $\kappa$) but $x^p+\pi^p \not\sim x^p$ (as one has finite $\tau_V$ and the other does not).}
\end{enumerate}
\end{remark}

We end this section with a discussion of elements in $V[[\underline{x}]]$ of order $2$, and how their lifts to $V[[\underline{x},y]]$ have a special form. From now until the end of this section, we assume that the characteristic of $\kappa$ is not $2$.

\begin{definition}\label{def:hessian}
    Given some $f \in V[[\underline{x}]]$, we define the \textbf{Hessian matrix of $f$} is the matrix defined as follows: we can write 
    \begin{equation*}
        \tilde{f}=v_0y^2+\sum_i \frac{v_i}{2} yx_i + \sum_{i \neq j} \frac{v_{i,j}}{2} x_i x_j + \sum_iv_{i,i}x_i^2 + \tilde{g},
    \end{equation*}
    \noindent for some units $v_0 , v_i, v_{i,j}$ for every $i,j$ and some $g \in \maxV^3$, as in Notation~\ref{def:subsoil}. Therefore we set
    \begin{equation*}
         \Hess(f) =
        \begin{bmatrix}
            v_{1,1} & \dots & v_{1, n} & v_1\\
            \vdots &  & \vdots  & \vdots \\
            v_{n,1} & \dots & v_{n,n} & v_n\\
            v_1 & \dots & v_n & v_0
        \end{bmatrix}
        \in \textup{Mat}_{n+1}(V).
    \end{equation*}
\end{definition}

\begin{remark}\label{rem:Hessian}
    \begin{enumerate}
        \item \textup{Note that the Hessian matrix $H=\Hess(f)$ satisfies the property that there exists some $g \in \maxV^3$ such that $\tilde{f}=\vec{x}H \vec{x}^\intercal +\tilde{g}$ where $\vec{x}=(x_1, \dots, x_n, y)$. } 
        \item \textup{We can view $ \Hess(f)$ as the matrix of second partial derivatives of $\tilde{f}$ at $\underline{0}$. That is, the $(i,j)-$th coordinate of $\Hess(f)$ is exactly $\frac{\partial^2(\tilde{f})}{ \partial x_i \partial x_j} (\underline{0},0) \in V$ (where $x_{n+1}=y$).  }
        \item \textup{We introduce the following notation: For every $r >0$, we define the ideal $\mathfrak{a}_r=\sum_{i=0}^{r-1} \langle \pi^i\rangle \cdot \langle x_1, \dots, x_n,y\rangle^{r-i} \subset V[[\underline{x},y]]$. For example, we have that  $\mathfrak{a}_3= \langle x_1, \dots, x_n, y \rangle^3 + \langle \pi \rangle \cdot \langle x_1, \dots, x_n, y \rangle^2 + \langle \pi^2 \rangle \cdot \langle x_1, \dots, x_n, y \rangle$ in $V[[ x_1, \dots, x_n, y ]]$. Intuitively, $g \in \mathfrak{a}_r$ if $g(\underline{0})=0$ and $g \in \naxV^r$. Note that $\mathfrak{a}_r$ is preserved under any automorphism of $V[[\underline{x},y]]$.} 
    \end{enumerate}
\end{remark}

The following proposition is an analogue over $V[[\underline{x}]]$ of the splitting lemma over field (see, for example, Theorem 2.47 of~\cite{greuel2007introduction} and Lemma 3.9 of~\cite{greuel2016right}), and it is a generalized analogue of Lemma 3.1 in~\cite{carvajal2019covers}. For more information on the splitting lemma over fields, see~\cite{greuel2025splitting}. In addition, its proof is inspired by the proof of Theorem 3.7 in~\cite{svoray2025ade}.

\begin{proposition}[Splitting Lemma]\label{prop:splitting}
Assume that the field $\kappa$ is quadratically closed and of characteristic different than $2$. Let $f \in V[[\underline{x}]]$ be of order $2$. Then there exists some $0 \leq k \leq n$ and some $g\left(x_{k+1}, \dots, x_n\right) \in \mathfrak{a}_3 \subset V[[x_{k+1}, \dots, x_n]]$  such that 
\begin{equation*}
    \tilde{f} \equivr y^2 + x_1^2 +\cdots + x_k^2 + g\left(x_{k+1}, \dots, x_n\right).
\end{equation*}  
\end{proposition}

\begin{proof}
    First, since $f \in \maxV^2 \setminus \maxV^3$ then we have that $\tilde{f} \in \naxV^2 \setminus \naxV^3$. So, following Notation~\ref{def:subsoil}, we can write $\tilde{f}=\sum_{\alpha} u_{\alpha} y^{n_{\alpha}} \underline{x}^{\alpha}$, and so we set $f_1=\sum_{|\alpha|=2} u_{\alpha} y^{n_{\alpha}} \underline{x}^{\alpha}$ and $h_1=\tilde{f}-f_1 \in \naxV^3$. \\

    Now, by Definition~\ref{def:hessian} we can write $f_1 = \vec{x} H \vec{x}^\intercal$ where $H=\Hess (f)$. Since $H$ is a symmetric matrix whose coefficients are either units or zero, then $f_1$ is a quadratic form, and so by Proposition 3.4 in Chapter I of~\cite{baeza2006quadratic} (noting that $2$ is invertible in $V$ since the characteristic of $\kappa$ is not $2$) there exists some invertible matrix $W$ and some diagonal matrix $D$ such that $H = W D W^\intercal$. Denote $\vec{w}=\vec{x}W=(w_1, \dots, w_{n+1})$. Note that since $W$ is invertible then by Lemma~\ref{lem:W[[x]]_good_prop} we can conclude that the $V-$homomorphism that sends $x_i \mapsto w_i$ for every $i \leq n$ and $y \mapsto w_{n+1}$ is a $V-$automorphism. In addition, up to rearranging the variables, we can denote the diagonal entries of $D$ by $u_1, \dots, u_k, u_{k+1}, \dots, u_{n+1}$ such that $u_1, \dots, u_k$ are units and $u_{k+1}, \dots, u_{n+1} \in \naxV$. Therefore we have that
    \begin{equation*}
        f_1 = \vec{x} H \vec{x}^\intercal = \vec{x} (W D W^\intercal) \vec{x}^\intercal = \sum_{i=1}^{n+1} u_i w_i^2.   
    \end{equation*}
    \noindent Now, since we can take a square root of $u_i$ for every $i \leq k$ (as $\kappa$ is quadratically closed and as $V[[\underline{x}]]$ is complete - and therefore Henselian) then
    \begin{equation*}
        \sum_{i=1}^{n+1} u_iw_i^2 + h_1 = \sum_{i=1}^{k} (v_iw_i)^2 + h_2,
    \end{equation*}
    where $v_i^2=u_i$ for every $i \leq k$ and $h_2 \in \naxV^3$. Therefore, up to $\equivr$ over $V[[\underline{x},y]]$, we can conclude that $\tilde{f} =  y^2 +  x_1^2 +\cdots  + x_k^2 + h_2$, where $h_2 \in \naxV^3$. Note that by the construction of automorphisms, we have that $h_2 \in \mathfrak{a}_3$. \\
    
    Now, we can write 
    \begin{equation*}
        \tilde{f} = y^2 +x_1^2 + \cdots + x_k^2 + f_3\left(x_{k+1}, \dots, x_n\right)  + yg_{n+1} + \sum_{i=1}^k x_i g_i,
    \end{equation*}
    \noindent for some $f_3 \in  \mathfrak{a}_3 \subset V[[x_{k+1}, \dots, x_n]]$ and some $g_0, g_1, \dots, g_k \in \naxV^2$.  Now, by looking at the automoprhism $\varphi$ defined by $\varphi(x_i)= x_i - \frac{g_i}{2}$ for $i=1,\dots, k$, $\varphi(y) = y - \frac{g_{n+1}}{2}$, and $\varphi(x_i)=x_i$ for every $i>k$, we have that 
    \begin{equation*}
        \varphi(\tilde{f}) = y^2 +x_1^2 + \cdots + x_k^2 + f_3(x_{k+1}, \dots, x_n) + f_4(x_{k+1}, \dots, x_n)  + yh_{n+1} + \sum_{i=1}^k x_i h_i,
    \end{equation*}
    \noindent where $f_4\left(x_{k+1}, \dots, x_n\right) \in \mathfrak{a}_4$ and $h_i \in \naxV^3$ for every $i$. Therefore, by repeating this process, we get that for every $l$ there exists an automorphism $\varphi_l$ such that
     \begin{equation*}
        \varphi_l(\tilde{f}) = y^2 +x_1^2 + \cdots + x_k^2 + \sum_{i=3}^l f_i\left(x_{k+1}, \dots, x_n\right) + y h_{l,n+1} + \sum_{i=1}^k x_i h_{l,i},
    \end{equation*}
    \noindent where $f_i \in \mathfrak{a}_i$ and $h_{l,i} \in \naxV^{l+1}$ for $i=0, \dots, k, n+1$. Thus if we denote $F_l=\varphi_l(f)$ for every $l$, then we have that
    \begin{equation*}
        F_{l+1}-F_l = f_{l+1} +  y(h_{l+1,n+1}-h_{l,n+1})+\sum_{i=1}^k x_i (h_{l+1,i}-h_{l,i}) \in \mathfrak{a}_l  \subset \naxV^l. 
    \end{equation*}
    \noindent Hence we can conclude that $\{F_l\}_{l=1}^\infty$ is a Cauchy sequence, and since $V[[\underline{x},y]]$ is complete, then it must converge to some $F$. Note that since $h_{l,i} \in \naxV^{l+1}$ for $i=0, \dots, k$, we must have that $h_l=  y h_{l,n+1} + \sum_{i=1}^k x_i h_{l,i} \to 0$ as $l \to \infty$. Therefore we can conclude that there exists some $g(x_{k+1}, \dots, x_n) \in \mathfrak{a}_3 \subset V[[x_{k+1}, \dots, x_n]]$  such that 
    \begin{equation*}
        F(x_1, \dots, x_n) =  y^2 +x_1^2 + \cdots + x_k^2  + g(x_{k+1}, \dots, x_n),
    \end{equation*}
    \noindent and since $F_l=\varphi_l(f)$ with $F_l \to F$, by Lemma~\ref{lem:equi_cont}, there exists some automorphism $\Phi$ such that $F=\Phi(f)$, which gives us that $\tilde{f} \equivr  y^2 + x_1^2 +\cdots + x_k^2 + g\left(x_{k+1}, \dots, x_n\right)$.  
\end{proof}

\begin{definition}
    The integer $k \in \mathbb{N}$ from Proposition~\ref{prop:splitting}  is called \textbf{the rank of $f$} and is denoted by $\text{rk}\left(f\right)$. The element $g \in \naxV^3$ is called \textbf{the residual part of $f$}. 
\end{definition}

\begin{lemma}
    If $f \sim g$ then $\textup{rk}(f) = \textup{rk}(g)$. 
\end{lemma}

\begin{proof}
    Observe that by Remark~\ref{rem:Hessian} and the proof of Propositions~\ref{prop:splitting} it is enough to show that the rank of the matrix $ \Hess(f)$ over $V$ is the same as that of $ \Hess(g)$. Since $f \sim g$ then there exists some $V-$automorpshim $\varphi$ of $V[[\underline{x},y]]$ and some unit $u \in V[[\underline{x},y]]$ such that $\tilde{f}\circ \varphi= u \tilde{g}$. Therefore, if we denote $\varphi(x_i)=\varphi_i$ for every $1 \leq i\leq n+1$ (where we write $x_{n+1}=y$), then for every $i,j$ we have that 
    \begin{equation*}
        \frac{\partial^2(\tilde{f} \circ \varphi)}{\partial x_i \partial x_j} (\underline{x}) = \sum_{k,l} \frac{\partial^2(\tilde{f})}{\partial x_k \partial x_l}(\varphi(\underline{x})) \cdot \frac{\partial \varphi_k}{\partial x_i}(\underline{x}) \cdot \frac{\partial \varphi_l}{\partial x_j}(\underline{x}) + \sum_k \frac{\partial \tilde{f}}{\partial x_k}(\varphi(\underline{x})) \cdot \frac{\partial^2(\varphi_k)}{\partial x_i \partial x_j} (\underline{x}).
    \end{equation*}
    Thus, if we denote by $D\varphi$ the Jacobian matrix of $\varphi$ at $\underline{0}$ (as in Lemma~\ref{lem:W[[x]]_good_prop}), then we can conclude that the Hessian matrix corresponding to $\tilde{f} \circ \varphi$ is $(D\varphi)^\intercal \cdot  \Hess(f) \cdot (D\varphi)$, which has the same rank as $ \Hess(\tilde{f})$ (as $D\varphi$ is invertible by Lemma~\ref{lem:W[[x]]_good_prop}). In addition, by the product rule, $ \Hess(g)$  has the same rank as the matrix corresponding to $u \tilde{g}$, which equals to $u(\underline{0}) \cdot  \Hess(g)$, where $u(\underline{0})$ satisfies $u - u(\underline{0}) \in \mathfrak{n}$. 
\end{proof}

The following lemma tells us that the residual part of $f$ is well defined up to equivalence. It is a mixed characteristic version of part 2 of Lemma 9.2.10 in~\cite{de2013local} and of part 1 of Theorem 2.1 in~\cite{greuel2025splitting}. 

\begin{lemma}\label{prop:rank_good_def}
    Let $f, g \in \mathfrak{a}_3 \subset V[[x_1, \dots, x_l,y]]$. Then we have  $f\left(x_1, \dots, x_l\right) \equivk  g\left(x_1, \dots, x_l\right)$ in $V[[x_1, \dots, x_l]]$ if and only if in $V[[\underline{x},y]]$ we have the equivalence $$f\left(x_1, \dots, x_l\right)+y^2+x_{l+1}^2 + \dots + x_n^2 \equivk  g\left(x_1, \dots, x_l\right)+y^2+x_{l+1}^2 + \dots + x_n^2.$$  
\end{lemma}

\begin{proof}
    First, if $f\left(x_1, \dots, x_l\right) \equivk g\left(x_1, \dots, x_l \right)$ then we can find some automorphism $\varphi$ of $V[[x_1, \dots, x_l, y]]$ and some unit $u \in V[[x_1, \dots, x_l, y]]$ such that $f = u \varphi(g)$. So, by extending $\varphi$ to an automorphism $\psi$ of $V[[\underline{x},y]]$ by setting $\psi \left(x_i\right)=v^{-1} x_i$ for every $i>l$, where $v^2=u$, we can conclude that 
    \begin{equation*}
        f\left(x_1, \dots, x_l\right) +x_{l+1}^2 + \dots + x_n^2  = u \psi\left(g\left(x_1, \dots, x_l\right) +x_{l+1}^2 + \dots + x_n^2\right). 
    \end{equation*}

    \noindent Second, assume that 
    \begin{equation*}
        F=f\left(x_1, \dots, x_l \right)+ y^2+ x_{l+1}^2 + \dots + x_n^2 \equivk g\left(x_1, \dots, x_l\right)+y^2 +x_{l+1}^2 + \dots + x_n^2=G.
    \end{equation*}
    Therefore, there exists some $\varphi_1, \dots, \varphi_{n+1} \in V[[\underline{x},y]]$ and some unit $u$ such that 
    \begin{equation*}
         g\left(\varphi_1, \dots, \varphi_{l}\right) + \sum_{i=l+1}^{n+1} \varphi_i^2 =   uf\left(x_1, \dots, x_l\right) + u\cdot \sum_{i=l+1}^n x_i^2 +uy^2.
    \end{equation*}
    
    \noindent Now,  write $\psi_i = u^{-0.5} \varphi_i = a_i + h_i$ where $a_i$ is a linear term over $\langle x_1, \dots, x_n,y \rangle$ and $h_i \in \naxV^2$. Then  we have that $$\sum_{i=l+1}^{n+1} a_i^2 - \sum_{i=l+1}^{n+1} x_i^2 - y^2 \in \naxV^2.$$ In addition, we have that 
    \begin{equation*}
        \sum_{i=l+1}^{n+1} \psi_i^2 = \sum_{i=l+1}^{n+1} a_i^2 + \sum_{i=l+1}^{n+1} h_i(2a_i+h_i),
    \end{equation*}
    and since $\ord_{\langle x_1, \dots, x_n,y \rangle}(a_{l+1}) =  \cdots =\ord_{\langle x_1, \dots, x_n,y \rangle}(a_{n+1})=1$,  we can conclude that $\sum_{i=1}^{n+1} a_i^2 = \sum_{i=l+1}^{n} x_i^2 +y^2$. Now, if we look at $F_i = 2a_i + h_i $ (recalling that the characteristic of $\kappa$ is not $2$) then  
    \begin{equation*}
        \det\left( \left[\frac{\partial F_i}{\partial x_j} \left(0\right)\right]_{n+1 \geq i,j \geq l+1}\right) =2\det\left( \left[\frac{\partial a_i}{\partial x_j} \left(0\right)\right]_{n+1 \geq i,j \geq l+1}\right) \in V^{\times},
    \end{equation*}
    \noindent where $x_{n+1}=y$. Therefore, by the implicit function theorem and the inverse function theorem (Lemma~\ref{lem:W[[x]]_good_prop}), there exists some $\theta_{l+1}, \dots, \theta_{n+1} \in V[[x_1, \dots, x_l]]$ for which we have that $F_i\left(x_1, \dots, x_l, \theta_{l+1}, \dots, \theta_{n+1}\right)=0$ for every $i$. For every $i \leq l$ define  $\beta_i = \psi_i\left(x_1, \dots, x_l, \theta_{l+1}, \dots, \theta_{n+1}\right)$. Then the map $\beta$ on $V[[x_1, \dots, x_l,y]]$ defined by $\beta\left(x_i\right)=\beta_i$ must be an automorphism. This tells us that $\left(\psi_1, \dots, \psi_l, a_{l+1}+\frac{h_{l+1}}{2}, \dots,  a_{n+1}+\frac{h_{n+1}}{2} \right)$ defines a $V-$automoprhism of $V[[\underline{x},y]]$. Thus, by the inverse function theorem (Lemma~\ref{lem:W[[x]]_good_prop}) we have that $\langle \psi_1, \dots, \psi_l, a_{l+1}+\frac{h_{l+1}}{2}, \dots,  a_{n+1}+\frac{h_{n+1}}{2} \rangle = \naxV$. Therefore, for every $i \leq l$ there exists some $\alpha_{1,i}, \dots,\alpha_{n+1,i}$ such that 
    \begin{equation*}
        x_i = \sum_{j=1}^l \alpha_{j,i} \psi_j + \sum_{j=l+1}^{n+1} \alpha_{j,i} \left(a_{j}+\frac{h_{j}}{2}\right), 
    \end{equation*}
    which gives us that 
    \begin{equation*}
        x_i=\sum_{j=1}^l \alpha_j\left(x_1, \dots, x_l, \theta_{l+1}, \dots, \theta_{n+1}\right) \beta_i. 
    \end{equation*}
    Therefore we have that $\langle \beta_1, \dots, \beta_l \rangle = \langle x_1, \dots, x_l \rangle$, and by the inverse function theorem (Lemma~\ref{lem:W[[x]]_good_prop}), we can conclude that $\beta(f)=f\left(\beta_1, \dots, \beta_l\right)= u\left(g\left(x_1, \dots, x_l\right)\right)$ and the result follows.
\end{proof}

\begin{remark}\label{rem:rk}
 \textup{Note that we can easily prove that the residual part of $f$ is well defined assuming $f$ has finite determinacy (as in Definition~\ref{def:det_V}), by applying Mather-Yau (Proposition~\ref{prop:mather-yau}). This is true since if $ f\left(x_1, \dots, x_l\right) + y^2 +x_{l+1}^2 + \cdots  + x_n^2 \sim  g\left(x_1, \dots, x_l\right)+ y^2 + x_{l+1}^2 + \cdots x_n^2$, then we have an isomorphism of $V$-ring 
\begin{equation*}
    \frac{V[[\underline{x},y]]}{J(\tilde{F})} \cong \frac{V[[\underline{x},y]]}{J(\tilde{G})}, 
\end{equation*}
and so after tensoring by $\frac{V[[\underline{x}]]}{\langle x_{l+1}, \dots, x_n\rangle}$ we can conclude that we get an isomorphism of $V-$algebras $\frac{V[[x_1, \dots, x_l,y]]}{J(\tilde{f})} \cong \frac{V[[x_1, \dots, x_l,y]]}{J(\tilde{g})}$. }
    %\item \textup{Note that if $u$ is a unit and $f \in \mathfrak{m}$ whose coefficients are multiplicative lifts, then $p^2+uf \sim p^2+f$. This is true since $p^2+uf=u\left(\left(u^{-0.5}p\right)^2+f\right) \sim \left(u^{-0.5}p\right)^2+f$, and so $\tilde{g}=\left(vy\right)^2+f$ for some unit $v \in V[[\underline{x},y]]$ where $g=\left(u^{-0.5}p\right)^2+f$ and we can apply the automorphism $y \mapsto vy$.}
\end{remark}

The following proposition gives us a classification of all $f \in V[[\underline{x}]]$ of full rank, which we can view as an analogue version of Morse's lemma over $V[[\underline{x}]]$, as presented in Theorem 2.46 of Chapter I in~\cite{greuel2007introduction} (which in turn is an analytic version of Morse's lemma, as presented in~\cite{ morse1934calculus, varchenko1985singularities}).  Note that it is also an analogue of Proposition 3.3. in~\cite{carvajal2019covers}, which can be thought of as a mixed characteristic version of Morse's lemma for surfaces, and its generalization in Proposition 3.14 of~\cite{svoray2025ade}. 

\begin{proposition} [Morse's Lemma]\label{prop:Morse}
    Let $f \in \mathfrak{m}^2 \subset V[[\underline{x}]]$, then the following are equivalent:
    \begin{enumerate}
        \item $\tau_V\left(f\right)=1$, 
        \item $\text{rk}\left(f\right)=n$, 
        \item $f \sim  \pi^2 + x_1^2 +\cdots + x_n^2$. 
    \end{enumerate}
\end{proposition}

\begin{proof}
Note that  we have that $\tau_V\left(f\right)=1$ if and only if $J(\tilde{f})+\langle \tilde{f},  a \rangle = \mathfrak{n}$. In addition, the equivalence of $2$ and $3$ follows by Proposition~\ref{prop:splitting} and the definition of $\textup{rk}$, as the residual part must be contained in $\mathfrak{a}_3$. Now, $3 \to 1$ is a direct computation and follows from the fact that equivalence preserves $\tau_V$ (Proposition~\ref{prop:order_equivalence}). Finally, for $1 \to 2$,  assume that $J(\tilde{f})+\langle \tilde{f}, a \rangle = \mathfrak{n}$ for some $a$ as in Proposition~\ref{prop:generic_intersection}. By Proposition~\ref{lem:app_ord_der}, we have that $f \in \mathfrak{m}^2 \setminus \mathfrak{m}^3$. Assume towards a contradiction that $\textup{rk}(f)=k < n$. Then $\tilde{f} \equivr y^2+  x_1^2 +\cdots +x_k^2 +g\left(x_{k+1}, \dots, x_n\right)$ for some $g \in \mathfrak{a}_3$, which must be non zero since $\tau_V\left(f\right) < \infty$. Yet, by direct computation we must have that $\frac{V[[\underline{x},y]]}{\langle \tilde{f} \rangle + J(\tilde{f})} \cong \frac{V[[x_{k+1}, \dots, x_n, y]]}{\langle g \rangle + J(g)}$ (where we view $g \in V[[x_{k+1}, \dots, x_n]]$), which is a contradiction as $\langle \partial_{k+1} (g) , \dots, \partial_{n} (g)\rangle \subset \naxV^2$, and so the multiplicity of $\frac{V[[x_{k+1}, \dots, x_n, y]]}{\langle g \rangle + J(g)}$ must be bigger than $1$. Therefore we must have that $k=n$. 
\end{proof}

\begin{remark}\label{rem:Morse_milnor}
\begin{enumerate}
    \item \textup{Note that if $\mu_V\left(f\right)=1$ then we must have that $\tau_V\left(f\right) \leq 1$, and since Remark~\ref{rem:smooth} tells us that $\mu_V\left(f\right)=0$ if and only if $\tau_V\left(f\right)=0$,  we must have that $\tau_V\left(f\right)=1$. The reverse direction, i.e., if $\tau_V\left(f\right)=1$ then $\mu_V\left(f\right)=1$ follows from steps $1 \to 2$ in Proposition~\ref{prop:Morse}. }
    \item \textup{We can prove a variant of Proposition~\ref{prop:Morse} based on Mather-Yau (as presented in Proposition~\ref{prop:mather-yau}). If we denote $f=\pi^2 + x_1^2 + \cdots + x_n^2$, then by Proposition~\ref{rem:determinacy} (or via similar computation to the one preformed in Example~\ref{ex:A_1_det}) we have that $f$ is $2-$determined. Therefore, given some $g \in V[[\underline{x}]]$ with $J(\tilde{g})=\langle x_1, \dots, x_n, y \rangle$, we have that $\frac{V[[\underline{x},y]]}{J(\tilde{f})} \cong \frac{V[[\underline{x},y]]}{J(\tilde{g})}$ and so $f \sim g$. }
    \item \textup{For a more general version of $2 \to 3$, given some $g \in \naxV^3$, since $f=\pi^2 + x_1^2 + \cdots + x_n^2$ is $2-$determined (by the previous item), we must have that $f+g \sim f$.}
    \item \textup{Inspired by Proposition~\ref{prop:Morse} one might want to preform an $ADE-$like classification up to $\sim$ over $\Vx$ based upon the the rank of order $2$ elements (as in Section 2.4 of chapter I in~\cite{greuel2007introduction} over $\mathbb{C}$, as in~\cite{greuel1990simple, Nguyen2013classification} over algebraically closed field of positive characteristic, or as in~\cite{carvajal2019covers, svoray2025ade} in the mixed characteristic case). Yet, in Proposition~\ref{prop:splitting} we get that the residual part is a member of $\mathfrak{a}_3$, and therefore it is unclear if it would be $\equivk$ to some $\tilde{h} \in \Vx$, which would give us an equivalence $\sim$ over $\Vx$. }
\end{enumerate}
\end{remark}

\section{Isolated Singularities over a DVR}\label{sec:Jeffries_Hochter}

In this section we show how the work of Hochster and Jefferies in~\cite{hochster2021jacobian}, Saito in~\cite{saito2022frobenius}, and KC in~\cite{kc2024singular} can inspire a similar construction to our Tjurina number that detects singularities over unramified and ramified DVRs, as opposed to Proposition~\ref{lem:tau(f,q)}. \\ 

We start with the unramified case. Specifically, we assume that $(V, \maxV, \kappa)$ is a complete DVR of mixed characteristic $(0,p)$ whose uniformizer is $p$ (i.e. $\maxV = \langle p \rangle$) and $\kappa$ is algebraically closed. In this case we want to quantify when $f \in \Vx$ defines an isolated singularity. We first review a few facts about $p-$derivation, as described in~\cite{hochster2021jacobian}. 

\begin{remark}
    \textup{By the Cohen structure theorem (see~\cite{cohen1946structure}), every complete DVR with uniformizer $p$ and an algebraically closed (therefore perfect) quotient field $\kappa$ must be isomorphic to the ring of Witt vectors over $\kappa$. For more about this connection between Witt vectors and Cohen rings, see~\cite{anscombe2022model} or Section 6 of~\cite{hazewinkel2009witt}, and for a deeper review of Witt vectors, see~\cite{rabinoff2014theory} or Chapter II, Section 6 of~\cite{serre1979local}.}
\end{remark}

For notational convention, we define the polynomial $$C_p\left(X_1, \dots, X_n\right) = \frac{X_1^p + \cdots +X_n^p - \left(X_1+\cdots +X_n\right)^p}{p} \in \mathbb{Z}[X_1, \dots, X_n].$$ Note that $C_p\left(X_1,X_2\right) \in \langle X_1X_2 \rangle
$ and, in general, for every integer $n$ we have that $C_p\left(X_1, \dots, X_n\right) \in \langle X_1, \dots, X_n \rangle^2$. \\

The following definition is a combination of Definition 2.2 and Definition 2.5 in~\cite{hochster2021jacobian}.

\begin{definition}
    A $p-$derivation is a map $\delta \colon \Vx \to \Vx$ such that for every $f, g \in \Vx$t:
    \begin{enumerate}
        \item $\delta\left(1\right)=\delta\left(0\right)=0$,
        \item $\delta\left(f+g\right)=\delta\left(f\right) + \delta\left(g\right) + C_p\left(f,g\right)$, 
        \item $\delta\left(f\cdot g\right) = f^p \delta\left(g\right) + g^p \delta\left(f\right) + p\delta\left(f\right) \delta\left(g\right)$,
        \item $\delta \mod    p^2 \colon \frac{\Vx}{\langle p^2 \rangle} \to \frac{\Vx}{\langle p \rangle}$ satisfies the conditions above.
    \end{enumerate}
\end{definition}

\begin{remark}
    \textup{ The notion of $p-$derivations goes back to the works of Buium in~\cite{buium1995differential} and J\'{o}yal in~\cite{joyal1985delta}. In addition, J\'{o}yal proved that the forgetful functor from the category of rings with $p-$derivations to the category of rings is the left adjoint of the Witt vector functor (with the natural $p-$derivation). One can view $p-$derivations as "derivations with respect to $p$", and play an analogous role to derivations, as we see with the mixed characteristic Jacobian criterion in Theorem~\ref{thm:J_smooth} or with the mixed characteristic Zariski-Nagata Theorem in~\cite{de2020zariski, de2025differential}. } 
\end{remark}

The following lemma summarizes a few results about $p-$derivations that are presented in Section 2 of~\cite{hochster2021jacobian}.

\begin{lemma}\label{lem:basic_delta}
    \begin{enumerate}
        \item For every $a \in V \subset \Vx$ we have that $\delta\left(a\right)=\frac{\Frob(a)-a^p}{p}$, where $\Frob$ is the unique lift of the Frobenius map to V (i.e. the unique endomorphism $\Frob$ on $V$ that satisfies $\Frob(a)\equiv a^p \mod  p$ for every $a\in V$). 
        \item Given $\mathfrak{a} \subset \Vx$  an ideal containing $p$, if $f \in \mathfrak{a}^{k+1}$ then $\delta\left(f\right) \in \mathfrak{a}^k$.
        \item For every $f_1, \dots, f_n$ there exists a unique $p$-derivation $\delta$ to $\Vx$ such that $\delta\left(x_i\right)=f_i$. 
        %\item For every $p-$derivation $\delta$ and for every $f \in \Vx$ we have that $\delta\left(pf\right) \equiv f^p \mod    p$ and $\delta\left(f^p\right) \equiv 0 \mod    p$.
    \end{enumerate}
\end{lemma}

\begin{remark}
\begin{enumerate}
    \item \textup{For more on the construction of $\Frob$ as the lift of the Frobenius map and for its properties, see~\cite{davis2014witt}.}
    \item \textup{The second item of Lemma~\ref{lem:basic_delta} can be viewed as a $p-$derivation analogue of Lemma~\ref{lem:app_ord_der}.}
\end{enumerate}
\end{remark}

Note that by Lemma~\ref{lem:basic_delta}, a $p-$derivation $\delta$ on $\Vx$ is uniquely defined by its values on $x_1, \dots, x_n$, and so we can associate the set of all $p-$derivation $\delta$ on $\Vx$ with $(\Vx)^{\times n}=V[[\underline{x}]] \times \cdots \times \Vx$ (together with the product $\maxV-$adic topology) via the association $\delta \mapsto (\delta(x_1), \dots, \delta(x_n))$. Specifically, given a point $\underline{a} \in (\Vx)^{\times n}$ we denote $\delta_{\underline{a}}$ the unique $p-$derivation that satisfies $\delta_{\underline{a}}(x_i)=a_i$ for every $i$.

\begin{proposition}\label{prop:continuous_delta}
    Given some $f\in \Vx$, then the map $(\Vx)^n \to \Vx$ defined by $\underline{a} \mapsto \delta_{\underline{a}}(f)$ is continuous (with respect to the the product $\maxV-$adic topology).
\end{proposition}

\begin{proof}
    Note that by item 2 of Lemma~\ref{lem:basic_delta} it is enough to prove the result for $f \in \Vx$ that are polynomial, as we can approximate $\delta_{\underline{a}}(f)$ by looking at the $\delta_{\underline{a}}$ of the finite Taylor expansions of $f$. In addition, since sum and product are continuous with respect to the $\maxV-$adic topology, it is enough to prove this result for monomials in $p, x_1, \dots, x_n$, which we prove via induction. The map $\delta \mapsto \delta(x_i)$ is continuous since it is simply the projection from the product topology. In addition, note that by Lemma~\ref{lem:basic_delta} we can conclude that $\delta(p^j)=\frac{p^j-p^{jp}}{p}$. Now, if $p^j\underline{x}^{\alpha}$ is a monomial (with $\alpha \neq \underline{0}$), we can write it as $p^j\underline{x}^{\alpha}=x_l \underline{x}^{\beta}$, and so 
    \begin{equation*}
    \delta(p^j\underline{x}^\alpha)=\delta(x_l p^j\underline{x}^\beta) = x_l^p \delta(p^j\underline{x}^\beta) + p^{jp} \underline{x}^{p\beta} \delta(x_l) + p\delta(x_l)\delta(p^j \underline{x}^{\beta}),
    \end{equation*}
    which is continuous as a combination of continuous functions. 
\end{proof}

The following proposition gives us a general formula for $\delta$ of some element in $V[[\underline{x}]]$, given the values $\delta(x_i)$ and the coefficients of $f$ as a power series. 

\begin{proposition}\label{prop:delta_general_form}
    For every $\delta$ and for every $f=\sum_{\alpha} a_{\alpha} \underline{x}^{\alpha}$. Then we have that 
    \begin{equation*}
        \delta(f)= \frac{\Frob_0(f) - f^p(\underline{x})}{p} + \sum_{\alpha} a_{\alpha}^p \delta(\underline{x}^{p\alpha}) + \sum_{\alpha} (\Frob(a_{\alpha}) - a_{\alpha}^p) \delta(\underline{x}^{\alpha}),
    \end{equation*}
   \noindent where $\Frob_0$ is the unique extension of $\Frob$ (as in item $1$ of Lemma~\ref{lem:basic_delta}) for $V$ to $\Vx$ by setting $\Frob_0(x_i)=x_i^p$ for every $i$.  In particular, $\delta_{\underline{0}}(f)=\frac{\Frob_0(f) - f^p(\underline{x})}{p}$ for every $f$. 
\end{proposition}

\begin{proof}
    By item 2 of Lemma~\ref{lem:basic_delta}, the function $\delta \colon \Vx \to \Vx$ is continuous (with respect to the $\maxV-$adic topology), and so it is enough to prove this result for polynomials $f$. Therefore, if we write the polynomial $f$ as $\sum_\alpha a_{\alpha}\underline{x}^\alpha$ for $a_\alpha \in V$, then we have that 
    \begin{equation*}
        \begin{split}
            \delta(f)&=\sum_{\alpha} \delta(a_{\alpha} \underline{x}^{\alpha}) + \frac{\sum_{\alpha} (a_{\alpha} \underline{x}^{\alpha})^p-f^p}{p}=\\
            &=\sum_{\alpha} (\underline{x}^{p\alpha}\delta(a_{\alpha}) + a_{\alpha}^p \delta(\underline{x}^{\alpha}) + p \delta(a_{\alpha})\delta(\underline{x}^{\alpha}) ) + \frac{\sum_{\alpha} (a_{\alpha} \underline{x}^{\alpha})^p-f^p}{p}.
        \end{split}
    \end{equation*}
    %$\delta(f)=\sum_{\alpha} \delta(a_{\alpha} \underline{x}^{\alpha}) + \frac{\sum_{\alpha} (a_{\alpha} \underline{x}^{\alpha})^p-f^p}{p}=\sum_{\alpha} (\underline{x}^{p\alpha}\delta(a_{\alpha}) + a_{\alpha}^p \delta(\underline{x}^{\alpha}) + p \delta(a_{\alpha})\delta(\underline{x}^{\alpha}) ) + \frac{\sum_{\alpha} (a_{\alpha} \underline{x}^{\alpha})^p-f^p}{p}$.
    \noindent Since $a_{\alpha} \in V$ then we have that $\delta(a_{\alpha})=\frac{\Frob(a_{\alpha})-a_{\alpha}^p}{p}$, and so by plugging this into the previous equation, the result follows. 
\end{proof}

We now recall the main result of~\cite{hochster2021jacobian} (in our case and using our notational convention) which we use in order to understanding and quantify isolated singularities using $p-$derivations.

\begin{theorem}[Theorem 4.9. in~\cite{hochster2021jacobian}]\label{thm:J_smooth}
    Let $f \in \Vx$ and let $f \in \mathfrak{q} \subset \Vx$ be a prime ideal. 
    \begin{enumerate}
        \item If $p \in \mathfrak{q}$ then $\left(\frac{\Vx}{\langle f \rangle}\right)_{\mathfrak{q}}$ is regular if and only if $\mathfrak{q}$ does not contain the ideal $\langle \left(\partial_1 f\right)^p, \dots, \left(\partial_n f\right)^p, \delta\left(f\right) \rangle$ for any $\delta$. 
        \item If $p \notin \mathfrak{q}$ then $\left(\frac{\Vx}{\langle f \rangle}\right)_{\mathfrak{q}}$ is regular if and only if $\mathfrak{q}$ does not contain the ideal $\langle \partial_1 f, \dots, \partial_n f \rangle$. 
    \end{enumerate}
\end{theorem}

\begin{remark}\label{rem:nagata}
    \textup{The second part Theorem~\ref{thm:J_smooth} was first presented in~\cite{nagata1958general}, and is in fact true for every mixed-characteristic complete DVR. }
\end{remark}

\begin{corollary}\label{cor:regular_length_zero}
    Given some $f \in \maxV$, we have that $\frac{\Vx}{\langle f \rangle}$ is regular  if and only if for any $p-$derivation $\delta$, the length of $\frac{\Vx}{\langle f,  \left(\partial_1 f\right)^p, \dots, \left(\partial_n f\right)^p, \delta\left(f\right) \rangle}$ over $\Vx$ is zero.
\end{corollary}

\begin{proof}
    Since $\Vx$ is local then we have that $\frac{\Vx}{\langle f \rangle} = \left(\frac{\Vx}{\langle f \rangle}\right)_{\maxV}$. Yet, since $f, p \in \maxV$ then by Theorem~\ref{thm:J_smooth}, $ \left(\frac{\Vx}{\langle f \rangle}\right)_{\maxV}$ is regular if and only if for every $\delta$ we have that $\langle \left(\partial_1 f\right)^p, \dots, \left(\partial_n f\right)^p, \delta\left(f\right) \rangle \not\subset \maxV$. But we can conclude that $\langle \left(\partial_1 f\right)^p, \dots, \left(\partial_n f\right)^p, \delta\left(f\right) \rangle \not\subset \maxV$ if and only if $\langle \left(\partial_1 f\right)^p, \dots, \left(\partial_n f\right)^p, \delta\left(f\right) \rangle = \Vx$, and since $0$ is the only module that has length zero, the result follows. 
\end{proof}

Inspired by Corollary~\ref{cor:regular_length_zero}, we can look at $\frac{\Vx}{\langle f,  \left(\partial_1 f\right)^p, \dots, \left(\partial_n f\right)^p, \delta\left(f\right) \rangle}$ and at its length as another version of the Tjurina number (that depends on the choice of $p-$derivation $\delta$). 

\begin{definition}\label{def:tjurina_delta}
    Let $f \in \Vx$. Then for every $p-$derivation $\delta$ we define the \textbf{Tjurina number of $f$ with respect to $\delta$}, denoted by $\tau (f, \delta)$, to be
    \begin{equation*}
        \tau(f, \delta) = \frac{1}{p^n} \cdot \textup{length}_{\Vx}\left(\frac{\Vx}{\langle f \rangle + J_\delta\left(f\right)}\right),
    \end{equation*}
     where $J_\delta\left(f\right) =\langle   \left(\partial_1 f\right)^p, \dots, \left(\partial_n f\right)^p, \delta\left(f\right) \rangle$. In addition, we define the \textbf{Milnor number of $f$ with respect to $\delta$} to be
    \begin{equation*}
        \mu(f, \delta) = \frac{1}{p^n} \cdot \textup{length}_{\Vx}\left(\frac{\Vx}{J_\delta\left(f\right)}\right).
    \end{equation*} 
\end{definition}

\begin{remark}\label{rem:77}
    \textup{The reason we divide by $p^n$ in the definition of $\tau(\cdot,\delta)$ and of $\mu(\cdot, \delta)$ is because since we are raising the derivatives of $f$ to the power of $p$, one would expect that length would be multiplied by $p$ as well. As we see in the Example~\ref{ex:delta_tjurina_A_1} at the end of this section, this gives us that the values need not be integers.}
\end{remark} 

\begin{example}
    \textup{The following example is inspired by Proposition~\ref{prop:splitting} (and therefore we assume that the characteristic of $\kappa$ is not $2$): Given so $p-$derivation $\delta$ on $\Vx$ such that $\delta(x_i) \in \maxV$ for every $i$, we show that}
    \begin{equation*}
        \mu(p^2 +x_1^2 +\cdots +x_k^2+g(x_{k+1}, \dots, x_n), \delta)= \mu_\kappa(\overline{g}),
    \end{equation*}
    \noindent\textup{where $\overline{g} = g \mod p$ and $\mu_\kappa(\overline{g})=\dim_\kappa\left( \frac{\kappa [[x_{k+1}, \dots, x_n]]}{J(\overline{g})}\right)$ is the Milnor number over $\kappa$.  Denote $h=p^2 +x_1^2 +\cdots +x_k^2+g(x_{k+1}, \dots, x_n)$, then we have that }
    \begin{equation*}
        J_\delta(h) = \langle x_1^p, \dots, x_k^p, (\partial_{k+1} g)^p, \dots, (\partial_n g)^p, \delta(h) \rangle.
    \end{equation*}
    \textup{Now, since $p^2=1+\cdots +1$ then $\delta(p^2)=\frac{p^2 - p^{2p}}{p}=p-p^{2p-1}$, and so we can conclude that }
    \begin{equation*}
        \begin{split}
            \delta(h) &= \delta(p^2 +x_1^2 +\cdots +x_k^2+g)=\\
            &= \delta(p^2) + \delta(x_1^2)+  \cdots + \delta(x_k^2) + \delta(g) + \frac{p^{2p} +x_1^{2p} +\cdots +x_k^{2p}+g^p - h^p}{p}=\\
            &= p\left(1-p^{2p-2} + \sum_{i=1}^k \delta(x_i)^2\right) + 2\left(\sum_{i=1}^k x_i^p \delta(x_i)\right) + \frac{p^{2p} +x_1^{2p} +\cdots +x_k^{2p}+g^p - h^p}{p}.
        \end{split}
    \end{equation*}
    \textup{Yet, we have that
    \begin{equation*}
        \frac{p^{2p} +x_1^{2p} +\cdots +x_k^{2p}+g^p - h^p}{p} = -\sum_{j=1}^{p-1} \frac{\binom{p}{j}}{p} p^{2j} g^{p-j} \mod \langle x_1^p, \dots, x_k^p \rangle.
    \end{equation*}
    \noindent As $\langle x_1^p, \dots, x_k^p \rangle \subset J_\delta(f)$, we can conclude that the ideal $J_\delta(f)$ is equal to  }
    \begin{equation*}
        \left\langle x_1, \dots, x_k, (\partial_{k+1} g)^p, \dots, (\partial_n g)^p,  p\left(1-p^{2p-2} + \sum_{i=1}^k \delta(x_i)^2 -\sum_{j=1}^{p-1} \frac{\binom{p}{j}}{p} p^{2j-1} g^{p-j}\right) \right\rangle,
    \end{equation*}
    \textup{and since $(1-p^{2p-2} + \sum_{i=1}^k \delta(x_i)^2 -\sum_{j=1}^{p-1} \frac{\binom{p}{j}}{p} p^{2j-1} g^{p-j})$ is a unit, we have that $J_\delta(h) = \langle x_1^p, \dots, x_k^p, (\partial_{k+1} g)^p, \dots, (\partial_n g)^p, p \rangle$. Thus we get that}
    \begin{equation*}
        \mu(h, \delta) = \frac{1}{p^n} \dim_\kappa \left(\frac{\kappa[[\underline{x}]]}{\langle x_1^p, \dots, x_k^p, (\partial_{k+1} g)^p, \dots, (\partial_n g)^p \rangle}\right).
    \end{equation*} 
    \noindent\textup{Now, if $v_1, \dots, v_\mu$ is a basis for the $\kappa-$vector space $\frac{\kappa[[x_{k+1}, \dots, x_n]]}{\langle  \partial_{k+1} g, \dots, \partial_n g \rangle}$, then we can conclude that the set $\{x_1^{l_1} \dots x_k^{l_k} v_j^{l_{k+1}} \colon 1 \leq j \leq \mu, 0 \leq l_i \leq p-1\}$ is a basis for the $\kappa-$vector space $\frac{\kappa[[\underline{x}]]}{\langle x_1^p, \dots, x_k^p, (\partial_{k+1} g)^p, \dots, (\partial_n g)^p \rangle}$, and so the result follows. }
\end{example}

The following result is an analogue of Proposition~\ref{lem:tau(f,q)}, and tells us that in fact $\tau(\cdot,\delta)$ completely detects isolated singularities, but with an additional condition.

\begin{proposition}\label{prop:finite_tjurina}
    Let $f \in \maxV$. Then $f$ defines an isolated singularity if and only if $p \in \sqrt{J\left(f\right)}$ and for any $p-$derivation $\delta$ we have that $\tau(f, \delta) < \infty$.
\end{proposition}

\begin{proof}
    Given some ideal $I \subset \Vx$, we have that the length of $\frac{\Vx}{I}$ is finite if and only if there exists some $k$ such that $\maxV^k \cdot \left(\frac{\Vx}{I}\right)=0$. This in turn is equivalent to having $\maxV^k \subset I$, which itself is equivalent having that if $\mathfrak{p}$ is a prime ideal such that $I \subset \mathfrak{p}$ then $\mathfrak{p}=\maxV$.
    Now, note that for every $\mathfrak{p} \neq \maxV$ we have that $\left(\frac{\Vx}{\langle f \rangle}\right)_\mathfrak{p} = 0$ (which is regular) if and only if $f \notin \mathfrak{p}$. Therefore,  $f$ defines an isolated singularity if and only if for every prime ideal $f \in \mathfrak{p} \neq \maxV$ we have that $\left(\frac{\Vx}{\langle f \rangle}\right)_\mathfrak{p}$ is a regular local ring.\\
    
    If $f$ defines an isolated singularity, by Theorem~\ref{thm:J_smooth}  for every $\delta$ and for every $\mathfrak{q} \neq \maxV$ we have that $J_\delta\left(f\right) + \langle f \rangle \not\subset \mathfrak{q}$ if $p \in \mathfrak{q}$ and $J\left(f\right) \not\subset \mathfrak{q}$ if $p \notin \mathfrak{q}$, and so $\tau(f, \delta) < \infty$. In addition, $p \in \sqrt{J\left(f\right)}$ since if $f$ defines an isolated singularity, then in particular, $\left(\frac{\Vx}{\langle f \rangle}\right)[\frac{1}{p}]$ is regular, and so as its completion with respect to $x_1, \dots, x_n$, which is isomorphic to $\frac{L[[\underline{x}]]}{\langle f \rangle}$, where $L$ is the fraction field of $V$. Since $L$ is a field of characteristic zero, then by Theorem 5.1.7. of~\cite{boubakri2009hypersurface} together with Lemma 2.44 in~\cite{greuel2007introduction} we must have that $1 \in J\left(f\right) \otimes_V L[[\underline{x}]]$, and so  $1 \in J\left(f\right)[\frac{1}{p}]$.\\
    
    For the other direction, assume that $p \in \sqrt{J\left(f\right)}$ and that $\tau(f, \delta) < \infty$ for every $\delta$. Then $J_\delta\left(f\right) + \langle f \rangle \not\subset \mathfrak{q}$ for every $\mathfrak{q} \neq \maxV$. So given $\mathfrak{q} \neq \maxV$, if $p \notin \mathfrak{q}$ then $J\left(f\right) \notin \mathfrak{q}$ since $p \in\sqrt{J\left(f\right)}$. If $p \in \mathfrak{q}$ then we have that $J_\delta\left(f\right) + \langle f \rangle \not\subset \mathfrak{q}$. Therefore by Theorem~\ref{thm:J_smooth} we have that $\left(\frac{\Vx}{\langle f \rangle}\right)_\mathfrak{q}$ is regular for every $\mathfrak{q} \neq \maxV$, and the result follows. 
\end{proof}

\begin{remark}
    \textup{An alternative proof of Proposition~\ref{prop:finite_tjurina}, similar to the proof of Proposition~\ref{lem:tau(f,q)}, follows from Theorem 3.23 in~\cite{de2020zariski}. It tells us that given some prime ideal $p \in \mathfrak{q} \subset \Vx$, then $\langle f \rangle + J_\delta\left(f\right) \subset \mathfrak{q}$ if and only if $f \in \mathfrak{q}^2$, and so we get that $\left(\frac{\Vx}{\langle f \rangle}\right)_{\mathfrak{q}}$ is not regular. }
\end{remark}

\begin{definition}
    Let $f \in \Vx$ define an isolated singularity. Then we denote by $\ord_{f}\left(p\right)$ the smallest $N$ such that $p^N \in \langle  f, \left(\partial_1 f\right)^p, \dots, \left(\partial_n f\right)^p \rangle$.  
\end{definition}

\begin{proposition}
    Let $f \in \Vx$ define an isolated singularity. Then we have that $\ord_{f}\left(p\right) \geq  \ord\left(f\right) -1$. 
\end{proposition}

\begin{proof}
First, by Proposition~\ref{prop:finite_tjurina} we have that $N=ord_{f}\left(p\right)$ is finite and well defined. Now, denoting $r=\ord\left(f\right)$, then from Lemma~\ref{lem:app_ord_der} we can conclude that $\langle \left(\partial_1 \left(f\right)\right)^p, \dots, \left(\partial_n \left(f\right)\right)^p \rangle   \subset \maxV^{p\left(r-1\right)} \subset \maxV^{r-1}$. Yet, as $p^N \in \langle  f, \left(\partial_1 f\right)^p, \dots, \left(\partial_n f\right)^p \rangle$,  there exists some $a \in \Vx$ such that $p^N + af \in \langle \left(\partial_1 \left(f\right)\right)^p, \dots, \left(\partial_n \left(f\right)\right)^p \rangle$. Therefore, as $f \in \maxV^r \subset \maxV^{r-1}$, we can conclude that $p^N \subset \maxV^{r-1} + \langle f \rangle \subset \maxV^{r-1}$, and the result follows since the order of $p \in \maxV$ is $1$. 
\end{proof}

\begin{proposition}\label{thm:constant_delta}
    Let $f \in \Vx$ and let $\underline{a} \in (\Vx)^{ \times n}$ such that $\mu(f, \underline{a}) < \infty$. Then there exists some $N$ such that for every $\underline{b} \in (\Vx)^{\times n}$ with $\underline{a} - \underline{b} \in (\maxV^N)^{\times n}=\maxV^N \times \cdots \times \maxV^N$ we have that $\tau(f, \delta_{\underline{a}}) \geq\tau(f, \delta_{\underline{b}})$ and $\mu(f, \delta_{\underline{a}}) \geq\mu(f, \delta_{\underline{b}})$. 
\end{proposition}

\begin{proof}
    If $\mu(f, \underline{a})<\infty$, then there exists some $N$ such that $\maxV^N \subset J_{\delta_{\underline{a}}}(f)$. Therefore, given some $\underline{b} =(b_1, \dots, b_n)$ such that  $b_i\in \maxV^N$ for every $i$, then $\delta_{\underline{b}}(f) \in \maxV^N$, and so $J_{\delta_{\underline{b}}}(f) \subset J_{\delta_{\underline{a}}}(f)$, and the result follows. 
\end{proof}

\begin{remark}
    \textup{From Proposition~\ref{thm:constant_delta} we can conclude that both $\mu(\cdot, \delta)$ and $\tau( \cdot, \delta)$ are semi-continuous with respect to $\delta$. This can be viewed as an analogue of Theorem 2.6 in~\cite{greuel2007introduction}, of Lemma A.13 in~\cite{greuel2016right}, and of item 1 in Remark~\ref{rem:unfolding}. }
\end{remark}

As mentioned in Remark~\ref{rem:77}, a problem with Proposition~\ref{prop:finite_tjurina} is that the value of $\tau(f, \delta)$ depends on the choice of $\delta$.  But, by Remark 4.3 in~\cite{hochster2021jacobian},  we have that the value $$\tau^\Delta\left(f\right) = \frac{1}{p^n} \textup{length}_{\Vx} \left(\frac{V [[\underline{x}]]}{\langle f, p \rangle + J_\delta\left(f\right)}\right)$$ does not depend on $\delta$. Note that  $\tau^\Delta\left(f\right) \leq \tau(f, \delta)$ for every $p-$derivation $\delta$.

\begin{example}
    \textup{Assume that $f=pg$ for some $g \in \Vx$. Then $\tau^\Delta(f)=\ord(\overline{g})$ if $n=1$ and $\tau^\Delta(f)=\infty$ otherwise, where $\overline{g} = g \mod p \in \kappa[[\underline{x}]]$. This is true since $J(f)=\langle p \rangle \cdot J(g)$ and $\delta(pg)=p^p \delta(g) + g^p \delta(p) + p \delta(g)\delta(p)$. As $\delta(p)$ is a unit,  we can conclude that $\langle f, p \rangle + J_\delta(f)=\langle pg, p , \delta(pg) \rangle + \langle p \rangle \cdot J(g) = \langle p , g^p \rangle$, and so }
    \begin{equation*}
        \tau^\Delta\left(f\right) = \frac{1}{p^n} \textup{length}_{\Vx} \left(\frac{V [[\underline{x}]]}{\langle p , g^p \rangle}\right) = \frac{1}{p^n} \dim_\kappa \left(\frac{\kappa[[\underline{x}]]}{ \langle \overline{g}^p \rangle} \right).
    \end{equation*}
\end{example}

The following lemma is a general result about local rings, inspired by Theorem 1.1. in~\cite{liu2018milnor} (and by Proposition 1 in~\cite{almiron2023tjurina}), which we use to bound $\tau(f, \delta)$. 

\begin{lemma}\label{lemma:Liu}
    Let $(R, \mathfrak{m})$ be a local ring, let $I$ be an $\mathfrak{m}-$primary ideal, and assume that $f \in R$ satisfies $f^N \in I$. Then we have that 
    \begin{equation*}
        \length_R\left(\frac{R}{I}\right) \leq N \cdot \length_R\left(\frac{R}{I+\langle f \rangle}\right), 
    \end{equation*}
    \noindent with equality if and only if $(I \colon \langle f\rangle) = \langle f^{N-1}\rangle + I$. 
\end{lemma}

\begin{proof}
    Define $\varphi \colon \frac{R}{I} \to \frac{R}{I}$ be the $R-$homomorphism defined by $\varphi(g + I) = fg + I$. Then, $\varphi$ induces a long exact sequence 
    \begin{equation*}
        0 \to \ker\left(\varphi\right) \to \frac{R}{I} \xrightarrow{\varphi} \frac{R}{I}  \to \frac{R}{I + \langle f \rangle} \to 0,
    \end{equation*}
    and by applying length over $R$ to this short exact sequence sequence we can conclude that $\textup{length}_{R}\left(\ker\left(\varphi\right)\right) = \textup{length}_{R}\left(\frac{R}{I + \langle f \rangle}\right)$. Since $f^N \in I$ we can conclude that $\langle f^{N-1} \rangle \cdot M \subset \ker\left(\varphi\right)$, where $M= \frac{R}{I}$. Therefore, we have a sequence of inclusions
    \begin{equation*}
        0 = \langle f^N \rangle \cdot  M \subset \langle f^{N-1} \rangle \cdot M \subset \cdots \subset \langle f \rangle \cdot M,
    \end{equation*}
    and so for every $i < N$ we have a short exact sequence 
    \begin{equation*}
        0 \to \ker\left(\varphi\right) \cap \langle f^i \rangle \cdot M \to \langle f^i \rangle \cdot M \xrightarrow{\varphi} \langle f^i \rangle \cdot M \to \frac{\langle f^i \rangle \cdot M }{ \langle f^{i+1} \rangle \cdot M },
    \end{equation*}
     which gives us that $\textup{length}_{R}\left(\frac{\langle f^i \rangle \cdot M }{ \langle f^{i+1} \rangle \cdot M }\right) = \textup{length}_{R}\left(\ker\left(\varphi\right) \cap \langle f^i \rangle \cdot M\right)$. But, since $\textup{length}_{\Vx}\left(\ker\left(\varphi\right) \cap \langle f^i \rangle \cdot M\right) =\textup{length}_{R}\left(\ker\left(\varphi\right)\right) = \textup{length}_{R}\left(\frac{R}{I + \langle f \rangle}\right)$, we can conclude that 
     \begin{equation*}
     \textup{length}_{R}\left(\frac{\langle f^i \rangle \cdot M }{ \langle f^{i+1} \rangle \cdot M }\right) \leq \textup{length}_{R}\left(\frac{R}{I + \langle f \rangle}\right). 
     \end{equation*} 
    \noindent Thus, we get that 
    \begin{equation*}
    \begin{split}
        \textup{length}_{R}\left(\frac{R}{I}\right) &=\textup{length}_R\left(M\right)=\\
        &=\textup{length}_{R}\left(\frac{R}{I + \langle f \rangle}\right) + \sum_{i=1}^{N-1} \textup{length}_{R}\left(\frac{\langle f^i \rangle \cdot M }{ \langle f^{i+1} \rangle \cdot M }\right) \leq\\
        &\leq N \cdot \textup{length}_{R}\left(\frac{R}{I + \langle f \rangle}\right), 
    \end{split}
    \end{equation*}
    Note that we have equality if and only if $\ker(\varphi) = \frac{\langle f^{N -1} \rangle +I}{I}$. This is true since we have an equality if and only if $\ker\left(\varphi\right) \cap \frac{\langle f^i \rangle+I}{I} = \ker\left(\varphi\right)$ for every $i$ (i.e. $\ker(\varphi) \subset \frac{\langle f^i \rangle+I}{I}$), but $ \frac{\langle f^{N-1} \rangle + I}{I} \subset \ker\left(\varphi\right)$. Yet, we have that $\ker(\varphi) = \{g+I \colon fg \in I \} = \frac{(I \colon \langle f \rangle)}{I}$, and the result follows. 
\end{proof}

% \begin{remark}
%     \textup{If we assume in Lemma~\ref{lemma:Liu} that $R$ is a domain, then we in fact have equality if and only if $N=1$ (i.e. $f \in I$). This is true since in this case we have that $(I \colon \langle f \rangle) = \frac{1}{f} (I \cap \langle f \rangle)$, and so }
% \end{remark}

\begin{corollary}\label{prop:upper_bound_tau}
    Let $f \in \maxV$. Then:
    \begin{enumerate}
        \item $\tau^\Delta \left(f\right) \leq \tau(f, \delta) \leq ord_{f}\left(p\right) \cdot \tau^\Delta\left(f\right)$ for every $p-$derivation $\delta$. 
        \item $f$ defines an isolated singularity if and only if $ p \in \sqrt{J\left(f\right)}$ and $\tau^\Delta\left(f\right) < \infty$.
        \item If there exists some $p-$derivation $\delta$ such that $\tau^\Delta\left(f\right) = \tau(f, \delta)$, then $f$ must be of order at most $2$. 
    \end{enumerate}
\end{corollary}

\begin{proof}
    The first part follows directly from Lemma~\ref{lemma:Liu}, and the second follows from the first together with Proposition~\ref{prop:finite_tjurina}. For the third part, we have the following short exact sequence:
    \begin{equation*}
        0 \to \frac{J_\delta\left(f\right) +\langle p \rangle}{J_\delta\left(f\right)} \to \frac{\Vx}{J_\delta\left(f\right)} \to \frac{\Vx}{J_\delta\left(f\right) + \langle p \rangle} \to 0. 
    \end{equation*}
    By looking at this short exact sequence, we have that $\tau^\Delta\left(f\right)=\tau(f, \delta)$ holds if and only if $\frac{J_\delta\left(f\right) +\langle p \rangle}{J_\delta\left(f\right)}=0$, which happens if and only if $p \in J_\delta\left(f\right)$. Yet, if $\ord\left(f\right)=r$ then by Lemma~\ref{lem:app_ord_der} and Lemma~\ref{lem:basic_delta} we have that $J_\delta\left(f\right) \subset \maxV^{r-1}$, and $p \in \maxV^{r-1}$ if and only if $r-1 \leq1$.
\end{proof}

\textup{There are a number of problems with looking at $p-$derivations to numerically detect isolated singularities:}
    \begin{itemize}
    \item \textup{The value of $\tau( \cdot, \delta)$ may (or may not) vary when setting different values for $\delta$.}
    \item \textup{Both $\tau(\cdot, \delta)$ and $\tau^\Delta$ are not preserved under $\sim$.}
    \item \textup{Unlike $\tau(\cdot,\delta)$, the finiteness of $\mu(f, \delta)$ depends on the choice of $\delta$.}
\end{itemize}

We see these problems in the following example:

\begin{example}\label{ex:delta_tjurina_A_1}
\textup{
Assume the characteristic of $\kappa$ is not $2$. Note that, $x^2 + p^2 \sim px$ in $V[[x]]$ as $x^2+y^2 = (x+iy)(x-iy)$ in $\Vxy$ (where $i^2=-1$, which exists since $\kappa$ is algebraically closed of characteristic different than $2$ and $V$ is complete - and therefore Henselian). For every $\delta$ we have that $J_\delta\left(px\right)=\langle p, x^p \rangle$ and so $\tau(px, \delta)=1= \tau^\Delta\left(px\right)$, but we also have that
\begin{equation*}
    J_\delta\left(x^2 + p^2\right) = \langle x^2 +p^2, x^p , p(-1 + p^{2p-2} + \delta\left(x\right)^2) + C_p\left(x^2,p^2\right) \rangle.
\end{equation*}
Yet, since $C_p\left(x^2,p^2\right) \in \langle p^2 x^2 \rangle$, then $C_p\left(x^2, p^2\right)=x^2p^2g$ for some $g \in V[[x]]$. Thus 
\begin{equation*}
    J_\delta\left(x^2 + p^2\right) = \langle x^2 +p^2, x^p , p(-1 + p^{2p-2} + px^2 g + \delta\left(x\right)^2) \rangle,
\end{equation*}
and we can conclude that  $\tau^\Delta\left(x^2+p^2\right)=\frac{2}{p}$. Therefore, for any for any $N <\ord_{f}\left(p\right)$ we can define $\delta_N\left(x\right)= \left(p^{N-1} + 1 - p^{2p-2} - px^2 g\right)^{0.5}$ (this square root exists since $p^{N-1} + 1 - p^{2p-2} - px^2 g$ is a unit), and we can conclude that $J_{\delta_N}\left(x^2 + p^2\right) = \langle x^2 +p^2, x^p, p^{N} \rangle$, giving different values for $\tau(f, \delta)$. In addition, for every $\delta\left(x\right) \in \maxV$ we have that $\mu(f, \delta) = 1$ (which is clearly finite). But for $\delta\left(x\right)=\left(-1+p^{2p-2}-px^2g +x\right)^{0.5}$ we have that $J_{\delta}\left(f\right)=\langle x^p, px\rangle$, and so $\frac{V[[x]]}{J_{\delta}\left(f\right)}$ is one-dimensional and therefore must have infinite length.}

\end{example}

%\begin{remark}
%    \textup{Note that It is easy to see that both $\mu_\delta$ and $\tau^\Delta$ are semicontinuous directly. That is, given some $f \in \Vx$ such that $\mu_{\delta}(f)<\infty$ for some $\delta$, there exists some $N$ such that if $\tilde{\delta}(x_i) \in \maxV^N$ for every $i$ then $\mu_{\tilde{\delta}}(f) \geq \mu_{\delta}(f)$ and $\tau_{\tilde{\delta}}(f) \geq \tau_{\delta}(f)$. This is true since if $\mu_{\delta}(f)<\infty$, then there exists some $N$ such that $\maxV^N \subset J_{\delta}(f)$. Therefore since $\tilde{\delta}(x_i) \in \maxV^N$ for every $i$ then $\tilde{\delta}(f) \in \maxV^N$ then $J_{\tilde{\delta}}(f) \subset J_{\delta}(f)$. Therefore we can conclude that $\mu_{\tilde{\delta}}(f) \geq \mu_{\delta_0}(f)$ and that $\tau_{\tilde{\delta}}(f) \geq \tau_{\delta}(f)$.  }
%\end{remark}

We end this section with an analogue of Proposition~\ref{prop:finite_tjurina} over a ramified DVR, i.e., a complete DVR $(V, \langle \pi \rangle, \kappa)$ of mixed characteristic $(0,p)$ such that $\kappa$ is algebraically closed and $p \in \langle \pi^2 \rangle$, based upon the work of KC in~\cite{kc2024singular}. Unlike the unramified case, the ramified case is much easier and can be directly compared with $\tau_V$. In particular, we have a special derivation that allows us to "derive with respect to $\pi$" directly. \\

We recall the main results in~\cite{kc2024singular} that are relevent for our discussion (specifically, Corollary 2.4 and Corollary 2.13). Note that in his paper, all results are over the ring of polynomials, but the results are valid over the ring of power series, as it is the completion of the ring of power series. It is a ramified analogue of Theorem~\ref{thm:J_smooth}.

\begin{proposition}\label{prop:ramified_J}
    \begin{enumerate}
        \item There exists a $\mathbb{Z}-$derivation $\pder \colon V[[\underline{x}]] \to \kappa[[\underline{x}]]$ such that $\pder(\pi)=1$. 
        \item Given some $f \in V[[\underline{x}]]$ and some prime ideal $\pi \in \mathfrak{q} \subset V[[\underline{x}]]$, then $\left( \frac{V[[\underline{x}]]}{\langle f \rangle} \right)_{\mathfrak{q}}$ is regular if and only if $\langle \partial_1(\overline{f}), \dots, \partial_n(\overline{f}), \pder(f) \rangle \not\subset \overline{\mathfrak{q}}$, where $\overline{\mathfrak{q}} = \mathfrak{q} \mod \pi \subset \kappax$ and $\overline{f} = f \mod \pi \in \kappax$. 
    \end{enumerate}
\end{proposition}

Inspired by Proposition~\ref{prop:ramified_J}, we can define an invariant that has an analogous result to that of Proposition~\ref{prop:finite_tjurina}. 

\begin{definition}\label{def:pitjurina}
    Let $f \in \maxV$. Then we define the \textbf{ramified Tjurina number of $f$} to be 
    \begin{equation*}
        \tau^\pi(f)=\dim_\kappa\left(\frac{\kappa[[\underline{x}]]}{\langle \overline{f} \rangle + J_\pi(f)} \right),
    \end{equation*}
    where $J_\pi(f) = \langle \partial_1(\overline{f}), \dots, \partial_n(\overline{f}), \pder(f) \rangle$.
\end{definition}

Therefore we can conclude the following Proposition, whose proof is the same as that of Proposition~\ref{prop:finite_tjurina}:

\begin{proposition}\label{prop:ramified_iso_sing}
    Let $f \in \maxV$. Then:
    \begin{enumerate}
        \item $\frac{V[[\underline{x}]]}{ \langle f \rangle}$ is regular if and only if $\tau^\pi(f)=0$
        \item $f$ defines an isolated singularities if and only if $\pi \in \sqrt{J\left(f\right)}$ and  $\tau^\pi\left(f\right) < \infty$.
    \end{enumerate}
\end{proposition}

\begin{proposition}\label{prop:inequality_tau_pi}
    Given some $f \in \maxV$. Then $\tau^\pi(f) \geq \tau_V(f)$.
\end{proposition}

\begin{proof}
    Let $f \in \maxV$ and let $\tilde{f}$ be as in Notation~\ref{def:subsoil}. Then we define $\partial_\pi(f)= \text{pr} (\partial_y(\tilde{f})) \in V[[\underline{x}]]$, where $\text{pr}$ is as in Remark~\ref{rem:y-p}. More explicitly, if $f=\sum_{\alpha} u_{\alpha} \pi^{n^{\alpha}} \underline{x}^{\alpha}$ as in Notation~\ref{def:subsoil} then $\partial_\pi(f)=\sum_{\alpha} u_{\alpha} n_{\alpha}\pi^{n_{\alpha}-1} \underline{x}^{\alpha}$. Since we have $\partial_\pi(f) \mod \pi =\pder(f)$, we get that $\tau^\pi(f)=\length_{V[[\underline{x}]]} \left( \frac{V[[\underline{x}]]}{J(f) + \langle \pi, f, \partial_\pi(f) \rangle} \right)$. Thus, the result follows from the third item of Proposition~\ref{prop:generic_intersection}. 
\end{proof}

\begin{remark}
\begin{enumerate}
    \item \textup{The proof of Proposition~\ref{prop:inequality_tau_pi} tells us that we can view $\tau^\pi$  as a ramified analogue of $\tau^\Delta$, since they both can be thought of as the co-length of $\langle f\rangle+J(f)$ plus the uniformizer and the respective "derivative with respect to the uniformizer". }
    \item \textup{Proposition~\ref{prop:inequality_tau_pi} gives us that if $f$ defines an isolated singularity then $\tau_V(f) <\infty$. This can be though of a generalization of Proposition~\ref{lem:tau(f,q)} in the ramified case. }
    \item \textup{As in Corollary~\ref{prop:upper_bound_tau}, we can use Lemma~\ref{lemma:Liu} and Proposition~\ref{prop:ramified_iso_sing} to show that if $f$ defines an isolated singularity then $\tau^\pi(f)$ is bounded above by $\ord_f(\pi) \cdot \length_{V[[\underline{x}]]} \left(\frac{V[[\underline{x}]]}{J(f) + \langle f, \partial_\pi(f) \rangle}\right)$, where $\ord_f(\pi)$ is the smallest $N$ such that $\pi^N \in J(f)+\langle f \rangle$. This allows us to "measure" when $f$ defines an isolated singularity by preforming a computation over $V[[\underline{x}]]$, which is analogous to the classic Tjurina number over a field. }
\end{enumerate}
\end{remark}

Yet, like $\tau(\cdot, \delta)$ and $\tau^\Delta$ (as in Example~\ref{ex:delta_tjurina_A_1},), the invariant $\tau^\pi$ is not preserved under $\sim$:

\begin{example}
    \textup{Let $f_1=x^2+\pi^p$ and $f_2=x^p +\pi^2$. Note that $f_1 \sim f_2$. Yet we have that $\overline{f_1}=x^2$ and $\overline{f_2}=x^p$ and so $\pder(f_1)=\pder(f_2)=0$.  Therefore $\tau^\pi(f_1)=\dim_\kappa\left(\frac{\kappa[[x]]}{\langle x \rangle}\right)=1$ and $\tau^\pi(f_2)=\dim_\kappa\left(\frac{\kappa[[x]]}{\langle x^p\rangle}\right)=p$. }
\end{example}

\bibliographystyle{alpha}
\bibliography{bibib}

\end{document}